\documentclass[a4paper,leqno,12pt]{amsart}
\usepackage{amsmath,amsthm}
\usepackage{amssymb}
\usepackage[all]{xypic}
\SelectTips{cm}{12}
\UseTips{}
\usepackage{euscript}
\begin{document}
%
%

\newenvironment{myeq}[1][]
{\stepcounter{thm}\begin{equation}\tag{\thethm}{#1}}
{\end{equation}}
\newenvironment{myeqn}[2][]
{\stepcounter{thm}\begin{equation}
\tag{\thethm}{#1}\vcenter{{#2}}}{\end{equation}}
\newcommand{\mydiag}[2][]{\myeq[#1]\xymatrix{#2}}
%
%
\newcommand{\mydiagram}[2][]
{\stepcounter{thm}\begin{equation}
    \tag{\thethm}{#1}\vcenter{\xymatrix{#2}}\end{equation}}
%
\newenvironment{mysubsection}[2][]
{\begin{subsec}\begin{upshape}\begin{bfseries}{#2.}
\end{bfseries}{#1}}
{\end{upshape}\end{subsec}}
\newenvironment{mysubsect}[2][]
{\begin{subsec}\begin{upshape}\begin{bfseries}{#2\vsn.}
\end{bfseries}{#1}}
{\end{upshape}\end{subsec}}
%
%
\newcommand{\sect}{\setcounter{thm}{0}\section}
\newcommand{\secta}{\setcounter{thm}{0}\section*}
\numberwithin{equation}{section}
\renewcommand{\theequation}{\thesection.\arabic{equation}}
\numberwithin{figure}{section}
\renewcommand{\thefigure}{\thesection.\arabic{figure}}
%
%
\theoremstyle{plain}
\swapnumbers
    \newtheorem{thm}{Theorem}[section]
    \newtheorem{rthm}[thm]{Resolution Theorem}
    \newtheorem{drthm}[thm]{Dual Resolution Theorem}
    \newtheorem{prop}[thm]{Proposition}
    \newtheorem{lemma}[thm]{Lemma}
    \newtheorem{hlemma}[thm]{Hauptlemma}
    \newtheorem{tlemma}[thm]{Transport Lemma}
    \newtheorem{cor}[thm]{Corollary}
    \newtheorem{fact}[thm]{Fact}
    \newtheorem{subsec}[thm]{}
    \newtheorem*{thma}{Theorem A}
    \newtheorem*{thmb}{Theorem B}
\theoremstyle{definition}
    \newtheorem{assume}[thm]{Assumption}
    \newtheorem{defn}[thm]{Definition}
    \newtheorem{example}[thm]{Example}
    \newtheorem{examples}[thm]{Examples}
    \newtheorem{claim}[thm]{Claim}
    \newtheorem{notn}[thm]{Notation}
    \newtheorem{conv}[thm]{Convention}
\theoremstyle{remark}
        \newtheorem{remark}[thm]{Remark}
    \newtheorem{ack}[thm]{Acknowledgements}
%
%
%
\newcommand{\xra}[1]{\xrightarrow{#1}}
\newcommand{\xla}[1]{\xleftarrow{#1}}
\newcommand{\xsim}{\xrightarrow{\sim}}
\newcommand{\hra}{\hookrightarrow}
\newcommand{\epic}{\to\hspace{-5 mm}\to}
\newcommand{\adj}[2]{\substack{{#1}\\ \rightleftharpoons \\ {#2}}}
\newcommand{\ccsub}[1]{\circ\sb{#1}}
\newcommand{\DEF}{:=}
\newcommand{\EQUIV}{\Leftrightarrow}
\newcommand{\hsp}{\hspace{10 mm}}
\newcommand{\hs}{\hspace*{5 mm}}
\newcommand{\hsm}{\hspace*{3 mm}}
\newcommand{\hsn}{\hspace*{2 mm}}
\newcommand{\vsm}{\vspace*{2 mm}}
\newcommand{\vsn}{\vspace{1 mm}}
\newcommand{\vs}{\vspace{4 mm}}
\newcommand{\vsp}{\vspace{7 mm}}
\newcommand{\rest}[1]{\lvert\sb{#1}}
\newcommand{\sotimes}{\overline{\otimes}}
\newcommand{\sbu}{\sb{\bullet}}
\newcommand{\ubu}{\sp{\bullet}}
\newcommand{\lra}[1]{\langle{#1}\rangle}
\newcommand{\llrra}[1]{\langle\langle{#1}\rangle\rangle}
\newcommand{\lo}[1]{\sb{({#1})}}
\newcommand{\up}[1]{\sp{({#1})}}
\newcommand{\upb}[1]{\sp{[{#1}]}}
\newcommand{\Fud}[2]{F\sp{#1}\lo{#2}}
\newcommand{\Gud}[2]{G\sp{#1}\lo{#2}}
\newcommand{\sud}[2]{\sigma\sp{#1}\lo{#2}}
%
%
\newcommand{\wA}{\widehat{A}}
\newcommand{\whC}{\widehat{C}\sb{\ast}}
\newcommand{\vare}{\varepsilon}
\newcommand{\es}[1]{e\sp{#1}}
\newcommand{\hmu}[1]{\mu\sp{#1}}
\newcommand{\tmu}[1]{\widetilde{\mu}\,\sp{#1}}
\newcommand{\thl}{\widetilde{\theta}\sp{L}}
\newcommand{\thr}{\widetilde{\theta}\sp{R}}
\newcommand{\thij}[2]{\widehat{\theta}\sp{({#1},{#2})}}
\newcommand{\htl}{\theta\sp{L}}
\newcommand{\htr}{\theta\sp{R}}
\newcommand{\htij}[2]{\theta\sp{({#1},{#2})}}
\newcommand{\rh}{\mathfrak{p}}
\newcommand{\wrh}[2]{\widetilde{\sigma}\sp{#1}\sb{#2}}
\newcommand{\di}[1]{\partial\sb{#1}}
\newcommand{\uP}[1]{{{#1}\sp{I}}}
\newcommand{\uPb}[1]{{({#1})\sp{I}}}
\newcommand{\uPn}[2]{{#2}\sp{I\sp{#1}}}
\newcommand{\uPnb}[2]{({#2})\sp{I\sp{#1}}}
\newcommand{\Ws}{W\sb{\ast}}
\newcommand\rH{{\rm H}}
\newcommand{\bN}{{\mathbb N}}
\newcommand\bR{{\mathbb R}}
\newcommand{\bZ}{{\mathbb Z}}
\newcommand{\bZO}{\bZ\sb{\otimes}}
\newcommand{\hZO}{\widehat{\bZ}\sb{\otimes}}
\newcommand{\bF}{\mathbb F}
\newcommand{\Fp}{\bF\sb{p}}
\newcommand{\Arr}{\operatorname{Arr}}
\newcommand{\Aut}{\operatorname{Aut}}
\newcommand{\colim}{\operatorname{colim}}
\newcommand{\comp}{\operatorname{comp}}
\newcommand{\Cone}[1]{\operatorname{Cone}({#1})}
\newcommand{\Ext}{\operatorname{Ext}}
\newcommand{\ho}{\operatorname{ho}}
\newcommand{\Ho}{\operatorname{Ho}}
\newcommand{\Hom}{\operatorname{Hom}}
\newcommand{\hHom}{\underline{\Hom}}
\newcommand{\uHom}{\underline{\mathbf{Hom}}}
\newcommand{\map}{\operatorname{map}}
\newcommand{\mapa}{\map\sb{\ast}}
\newcommand{\mapM}{\map\sb{\eM}}
\newcommand{\mapt}{\mathbf{map}\sb{\ast}}
\newcommand{\maps}{\mathbf{map}\sb{\Sp}}
\newcommand{\Map}[1]{\mbox{\textbf{map}}\sb{#1}}
\newcommand{\MapC}{\Map{\eC}}
\newcommand{\MapCp}{\Map{\eC'}}
\newcommand{\MapK}{\Map{K}}
\newcommand{\MapL}{\Map{L}}
\newcommand{\Mor}{\operatorname{Mor}}
\newcommand{\sk}[1]{\operatorname{sk}\sb{#1}}
\newcommand{\Obj}{\operatorname{Obj}\,}
%
%
\newcommand{\eA}{\EuScript A}
\newcommand{\eC}{\EuScript C}
\newcommand{\eD}{\EuScript D}
\newcommand{\eI}{\EuScript I}
\newcommand{\eL}{\EuScript L}
\newcommand{\eM}{\EuScript M}
\newcommand{\eS}{\EuScript S}
\newcommand{\Sa}{\eS\sb{\ast}}
\newcommand{\dK}{\mathcal{K}}
\newcommand{\dL}{\mathcal{L}}
\newcommand{\dZ}{\mathcal{Z}}
\newcommand{\op}{\sp{\operatorname{op}}}
%
%
\newcommand{\Ab}{\mbox{\sf Ab}}
\newcommand{\AbGp}{\mbox{\sf AbGp}}
\newcommand{\Cat}{\mbox{\sf Cat}}
\newcommand{\Ch}{\mbox{\sf Ch}}
\newcommand{\ChZ}{\Ch\sb{\bZ}}
\newcommand{\ChR}{\Ch\sb{R}}
\newcommand{\ChRz}{\ChR\sp{\geqslant 0}}
\newcommand{\ModR}{\mbox{\sf Mod}\sb{R}}
\newcommand{\grM}{\mbox{\sf grMod}\sb{R}}
\newcommand{\grMz}{\grM\sp{\geqslant 0}}
\newcommand{\Set}{\mbox{\sf Set}}
\newcommand{\Seta}{\Set\sb{\ast}}
\newcommand{\Sp}{\mbox{\sf Sp}}
\newcommand{\Top}{\mbox{\sf Top}}
\newcommand{\Topa}{\Top\sb{\ast}}
\newcommand{\squa}{\pmb \square}
\newcommand{\LBox}{\overline{\squa}}
\newcommand{\tria}{\pmb \triangle}
\newcommand{\hy}[2]{{#1}\text{-}{#2}}
\newcommand{\MCat}{\hy{\eM}{\Cat}}
%
%
\newcommand{\nul}[1]{\operatorname{nul}\sb{#1}}
\newcommand{\nnC}{\nul{n}\eC}
\newcommand{\Nul}[1]{\operatorname{Nul}\sb{#1}}
\newcommand{\NnC}{\Nul{n}\eC}
\newcommand{\NlnC}{\Nul{\leq n}\eC}
\newcommand{\II}[2]{I\sp{#1}\sb{#2}}
\newcommand{\Ic}[1]{{\mathcal I}\sp{#1}}
\newcommand{\li}[1]{\sb{#1}}
\newcommand{\lb}[1]{\sb{({#1})}}
\newcommand{\uv}{\lb{u,v}}
%
%
\newcommand{\wh}{~--\ }
\newcommand{\wwh}{--\ }
\newcommand{\w}[2][ ]{\ \ensuremath{#2}{#1}\ }
\newcommand{\ww}[1]{\ \ensuremath{#1}}
\newcommand{\www}[2][ ]{\ensuremath{#2}{#1}\ }
\newcommand{\wwb}[1]{\ \ensuremath{(#1)}-}
\newcommand{\wb}[2][ ]{\ (\ensuremath{#2}){#1}\ }
\newcommand{\wbb}[2][ ]{(\ensuremath{#2}){#1}\ }
\newcommand{\wref}[2][ ]{\ \eqref{#2}{#1}\ }
%
%
\newcommand{\Id}{\operatorname{Id}}
\newcommand{\inc}{\operatorname{inc}}
\newcommand{\Image}{\operatorname{Im}}
\newcommand{\Ker}{\operatorname{Ker}}
%
%
\newcommand{\cE}{\mathcal{E}\sb{\ast}}
\newcommand{\cF}{\mathcal{F}\sb{\ast}}
\newcommand{\cG}{\mathcal{G}\sb{\ast}}
\newcommand{\ccH}{\mathcal{H}\sb{\ast}}
\newcommand{\cI}{\mathcal{I}}
\newcommand{\cK}{\mathcal{K}}
\newcommand{\cM}{\mathcal{M}}
\newcommand{\cN}{\mathcal{N}}
\newcommand{\cV}{\mathcal{V}\sb{\ast}}
\newcommand{\cT}{\mathcal{T}}
%
%
\newcommand{\n}{{\mathbf{\langle n \rangle}}}
\newcommand{\pn}{{\mathbf{P\sp{n}}\sb{\ast}}}
\newcommand{\pnp}{{\mathbf{P\sp{n+1}\sb{\ast}}}}
\newcommand{\bA}{{\mathbf{A}\sb{\ast}}}
\newcommand{\bB}{{\mathbf{B}\sb{\ast}}}
\newcommand{\bC}{{\mathbf{C}\sb{\ast}}}
\newcommand{\bD}{{\mathbf{D}\sb{\ast}}}
\newcommand{\bK}{{\mathbf{K}\sb{\ast}}}
\newcommand{\bL}{{\mathbf{L}\sb{\ast}}}
\newcommand{\bM}[2]{\mathbf{M}({#1},{#2})\sb{\ast}}
\newcommand{\tMn}[3]{\widetilde{\mathbf{M}}({#1},{#2})\sb{#3}}
\newcommand{\tM}[2]{\tMn{#1}{#2}{\ast}}
\newcommand{\Mt}[2]{\mathbf{M}({#1},{#2})}
\newcommand{\bX}{{\mathbf{X}\sb{\ast}}}
%
%
\newcommand{\fn}[3]{\sb{#1}\sp{{#2}{#3}}}
\newcommand{\znz}{\fn{n}{0}{0}}
\newcommand{\znb}{\fn{n}{0}{1}}
\newcommand{\bnz}{\fn{n}{1}{0}}
\newcommand{\bnb}{\fn{n}{1}{1}}
\newcommand{\znmz}{\fn{n-1}{0}{0}}
\newcommand{\znmb}{\fn{n-1}{0}{1}}
\newcommand{\bnmz}{\fn{n-1}{1}{0}}
\newcommand{\bnmb}{\fn{n-1}{1}{1}}
\newcommand{\znmmz}{\fn{n-2}{0}{0}}
\newcommand{\znmmb}{\fn{n-2}{0}{1}}
\newcommand{\bnmmz}{\fn{n-2}{1}{0}}
\newcommand{\bnmmb}{\fn{n-2}{1}{1}}
\newcommand{\znpz}{\fn{n+1}{0}{0}}
\newcommand{\znpb}{\fn{n+1}{0}{1}}
\newcommand{\bnpz}{\fn{n+1}{1}{0}}
\newcommand{\bnpb}{\fn{n+1}{1}{1}}
\newcommand{\znppz}{\fn{n+2}{0}{0}}
\newcommand{\znppb}{\fn{n+2}{0}{1}}
\newcommand{\bnppz}{\fn{n+2}{1}{0}}
\newcommand{\bnppb}{\fn{n+2}{1}{1}}
%
%
\newcommand{\as}[2]{{#1}\sb{#2}}
\newcommand{\Qz}{\as{Q}{0}}
\newcommand{\Qb}{\as{Q}{1}}
\newcommand{\s}{\sharp}
\newcommand{\sH}{\s(H)}
\newcommand{\sE}{\s(E)}
\newcommand{\Sprime}{S\sp{\prime}}
\newcommand{\Tprime}{T\sp{\prime}}
%
%
\newcommand{\pzz}{\sp{00}}
\newcommand{\pzb}{\sp{01}}
\newcommand{\pbz}{\sp{10}}
\newcommand{\pbb}{\sp{11}}
\newcommand{\pST}{\sp{S,T}}
\newcommand{\pSpTp}{\sp{\Sprime,\Tprime}}
%
%
\newcommand{\parm}[2]{\sb{#1}\sp{#2}}
\newcommand{\na}{\parm{n}{\bA}}
\newcommand{\nma}{\parm{n-1}{\bA}}
\newcommand{\nmma}{\parm{n-2}{\bA}}
\newcommand{\npa}{\parm{n+1}{\bA}}
\newcommand{\nb}{\parm{n}{\bB}}
\newcommand{\nmb}{\parm{n-1}{\bB}}
\newcommand{\npb}{\parm{n+1}{\bB}}
\newcommand{\nppb}{\parm{n+2}{\bB}}
\newcommand{\nc}{\parm{n}{\bC}}
\newcommand{\nmc}{\parm{n-1}{\bC}}
\newcommand{\npc}{\parm{n+1}{\bC}}
\newcommand{\nppc}{\parm{n+2}{\bC}}
\newcommand{\npppc}{\parm{n+3}{\bC}}
\newcommand{\nd}{\parm{n}{\bD}}
\newcommand{\nmd}{\parm{n-1}{\bD}}
\newcommand{\npd}{\parm{n+1}{\bD}}
\newcommand{\nppd}{\parm{n+2}{\bD}}
\newcommand{\npppd}{\parm{n+3}{\bD}}
\newcommand{\pih}[2]{[{#1},\,{#2}]}
\newcommand{\susp}{{\mathbf{\Sigma}}}
\newcommand{\tsusp}{\widetilde{\susp}}
\newcommand{\lop}[1]{\widetilde{\Omega}\sp{#1}}
%
%
%
\title{Higher Toda brackets and Massey products}
\author[H.-J.~Baues]{Hans-Joachim Baues}
\address{Max-Planck-Institut f\"{u}r Mathematik\\
Vivatsgasse 7\\ 53111 Bonn, Germany}
\email{baues@mpim-bonn.mpg.de}
\author[D.~Blanc]{David Blanc}
\address{Department of Mathematics\\ University of Haifa\\ 31905 Haifa, Israel}
\email{blanc@math.haifa.ac.il}
\author[S.~Gondhali]{Shilpa Gondhali}
\address{Department of Mathematics\\ University of Haifa\\ 31905 Haifa, Israel}
\email{sgondhal@math.haifa.ac.il, shilpa.s.gondhali@gmail.com}
\date{\today}

\begin{abstract}
We provide a uniform definition of higher order Toda brackets in a general setting,
covering the known cases of long Toda brackets for topological spaces and Massey
products for differential graded algebras, among others.
\end{abstract}
\maketitle

\setcounter{section}{0}

%
%
\section*{Introduction}
\label{cint}

Toda brackets and Massey products have played an important role in homotopy
theory ever since they were first defined in \cite{MassN} and \cite{TodG,TodC}:
in applications, such as \cite{AdHI,BJMahT,MPetS}, and in a more theoretical
vein, as in \cite{AdSS,BauA,HellSH,KristS,MargS,SagaU,SpanS}.
There are a number of variants (see, e.g., \cite{JCAlexC,HKMarC,MizuM,PSteS} and
\cite[\S 3.6.4]{BauO}), as well as higher order versions including
\cite{KlauT,KraiM,KMadE,MaunC,MoriHT,GPorW,GPorH,RetaL,SpanH,GWalkL}.
In recent years they have appeared in many other areas of mathematics,
including symplectic geometry, representation theory, deformation theory,
topological robotics, number theory, mathematical physics, and algebraic geometry
(see \cite{BTaimM,BKSchwR,FWeldM,GranT,KimM,LaurTB,LStreL,RizzI}).

Toda brackets were originally defined for diagrams of the form
\begin{myeq}\label{eqtodabrackets}
S\sp{n}~\xra{f}~S\sp{p}~\xra{g}~S\sp{k}~\xra{h}~X~,
\end{myeq}
\noindent with \w{g\circ f} and \w{h\circ g} nullhomotopic.

If we choose nullhomotopies  \w{F:g\circ f\sim 0} and
\w[,]{G:h\circ g\sim 0} they fit into a diagram of cones as in
Figure \ref{figtb}:

\setcounter{figure}{\value{thm}}\stepcounter{subsection}
\begin{figure}[htbp]
\begin{center}
%
%
\begin{picture}(367,130)(-20,-10)
%
%
\put(-13,46){$S^{n+1}$}
\put(20,0){\line(0,1){45}}
\put(25,0){\oval(10,15)[bl]}
\put(15,45){\oval(10,15)[tr]}
\put(20,60){\line(0,1){40}}
\put(25,100){\oval(10,15)[tl]}
\put(15,60){\oval(10,15)[br]}
%
%
\put(44,46){{\scriptsize $S^{n}$}}
\put(30,50){\line(1,3){20}}
\put(70,50){\line(-1,3){20}}
\put(42,66){{\scriptsize $CS^{n}$}}
\bezier{150}(30,50)(55,70)(70,50)
\bezier{150}(30,50)(45,30)(70,50)
\put(30,50){\line(1,-3){20}}
\put(70,50){\line(-1,-3){20}}
\put(42,24){{\scriptsize $CS^{n}$}}
\put(75,50){\vector(1,0){50}}
\put(95,55){{\scriptsize $f$}}
\put(67,95){\vector(1,0){70}}
\put(95,100){{\scriptsize $Cf$}}
\put(75,5){\vector(4,1){160}}
\put(160,17){{\scriptsize $F$}}
%
%
\put(149,46){{\scriptsize $S^{p}$}}
\put(135,50){\line(1,3){20}}
\put(175,50){\line(-1,3){20}}
\put(147,65){{\scriptsize $CS^{p}$}}
\bezier{150}(135,50)(155,70)(175,50)
\bezier{150}(135,50)(155,30)(175,50)
\put(185,50){\vector(1,0){45}}
\put(200,55){{\scriptsize $g$}}
\put(170,95){\vector(4,-1){137}}
\put(245,80){{\scriptsize $G$}}
%
%
\bezier{150}(237,50)(253,67)(263,50)
\bezier{150}(237,50)(247,33)(263,50)
\put(245,46){{\scriptsize $S^{k}$}}
\put(268,50){\vector(1,0){35}}
\put(285,53){{\scriptsize $h$}}
%
%
\bezier{150}(310,50)(326,67)(336,50)
\bezier{150}(310,50)(322,33)(336,50)
\put(321,47){{\scriptsize $X$}}

\end{picture}
\caption{\label{figtb}The Toda bracket construction}
\end{center}
\end{figure}
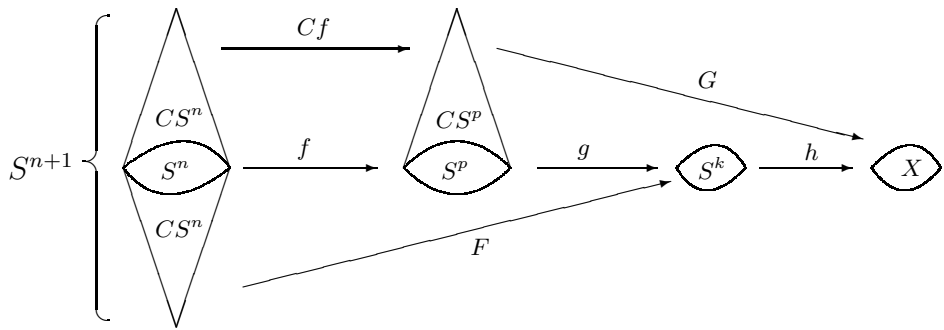
\setcounter{thm}{\value{figure}}

This yields an element \w{\lra{h,g,f}} in \w[,]{[S\sp{n+1},\,X]} called the
\emph{Toda bracket}. The value we get depends on the choices of nullhomotopies 
$F$ and $G$, so it is not uniquely determined.
The Toda bracket is thus more properly a certain double coset of
\w[.]{h\sb{\#}\pi\sb{n+1}(S\sp{k})+\Sigma f\sp{\#}\pi\sb{p+1}(X)}

If we view \w{[h]} as an element in \w[,]{\pi\sb{\ast}X} while \w{[g]}
is seen as a primary homotopy operation acting trivially on \w{[f]} and
\w{[h]\circ[g]=0} is a relation among primary operations, we can think of the Toda
bracket as a secondary homotopy operation.
Similarly, a diagram of the form
\begin{myeq}\label{eqtodabracketc}
X~\xra{f}~K(G,n)~\xra{g}~K(G',p)~\xra{h}~K(G'',k)
\end{myeq}
\noindent with \w{g\circ f\sim 0 \sim h\circ g} defines a secondary cohomology
operation in the sense of \cite{AdHI}.

On the other hand, the \emph{Massey product} in cohomology \wh
defined whenever we have three classes \w{\alpha,\beta,\gamma\in H\sp{\ast}X}
with \w{\alpha\cdot\beta=0=\beta\cdot\gamma} \wwh is a different type of
secondary cohomology operation which does not fit into this paradigm.

All three examples have higher order versions, though the precise definitions
are not always self-evident or unique (cf.\ \cite{GWalkL} and \cite{MaunC,KlauT}).
Nevertheless, these higher order operations play an
important role in homotopy theory \wh for instance, in enhancing our theoretical
understanding of spectral sequences (cf.\ \cite{BBlaC}) and in providing a
conceptual full invariant for homotopy types of spaces (see \cite{TanrH} and
\cite{BJTurHA}).

The main goal of this note is to explain that higher order Toda brackets and
higher Massey products have a uniform description, covering all cases known to the
authors (including both the homotopy and cohomology versions).

The setting for our general notion of higher Toda brackets is any category
$\eC$ enriched in a suitable monoidal category $\eM$. In fact, the
minimal context in which higher Toda brackets can be defined is just an enrichment
in a monoidal category equipped with a certain structure of ``null cubes'',
encoded by the existence of an augmented path space functor \w{PX\to X}
satisfying certain properties (abstracted from those enjoyed by the usual
path fibration of topological spaces). We call such an $\eM$ a
\emph{monoidal path category} \wh see Section \ref{cpfmc}.

In this context we can define the notion of a higher order chain complex:
that is, one in which the identity \w{\partial\partial=0} holds only up to a
sequence of coherent homotopies (see Section \ref{chocc}). This suffices
to allow us to \emph{define} the values of the corresponding higher order
Toda bracket (see Section \ref{chtb}, where higher Massey products are also
discussed).

However, in order for these Toda brackets to enjoy the expected properties,
such as homotopy invariance, $\eM$ must be also be a simplicial model category.
In this case there is a model category structure on the category \w{\MCat} of
categories enriched in $\eM$, due to Lurie, Berger and Moerdijk, and others,
in which the weak equivalences are Dwyer-Kan equivalences (see \S \ref{ddke}).
This is explained in Section \ref{chtbmc}, where we prove:

%
%
\begin{thma}
Higher Toda brackets are preserved under Dwyer-Kan equivalences.
\end{thma}
\noindent [See Theorem \ref{tdke} below].

We also show that the usual higher Massey products in a differential graded algebra
correspond to our definition (see Proposition \ref{pmassey}).

In Section \ref{ctbcc} we study the case of ordinary
Toda brackets for chain complexes, and show their interpretation as
secondary $\Ext$-operations.

\begin{notn}\label{snac}
The category of sets will be denoted by \w[,]{\Set} that of compactly
generated topological spaces by \w{\Top} (cf.\ \cite{SteCC}, and
compare \cite{VogtCC}), and that of pointed compactly generated spaces
by \w[.]{\Topa}

If $R$ is a commutative ring with unit, the category of $R$-modules will be
denoted by \w{\ModR} (though that of abelian groups will be denoted simply by
\w[).]{\AbGp} The category of non-negatively graded $R$-modules will be denoted by
\w[,]{\grMz} with objects \w[,]{\cE=\{E\sb{n}\}\sb{n\geq 0}} and so on.

The category of $\bZ$-graded chain complexes over \w{\ModR} will be denoted
by \w[,]{\ChR} with objects \w[,]{\bA} \w[,]{\bB} and so on, where
$$
\bA~\DEF~(\ldots A_{n}\xrightarrow{\partial_n} A_{n-1}
\xrightarrow{\partial_{n-1}}A_{n-2}
\xrightarrow{\partial_{n-2}} A_{n-3}\ldots)~.
$$
\noindent The category of nonnegatively graded chain complexes over \w{\ModR}
will be denoted by \w[.]{\ChRz}
A chain map \w{f:\bA\to\bB} inducing an isomorphism
\w{f\sb{\ast}\colon H\sb{n}\bA \longrightarrow H\sb{n}\bB} for all $n$ is
called a \emph{quasi-isomorphism}.

Finally, the category of simplicial sets will be denoted by $\eS$, and that of
pointed simplicial sets by \w[.]{\Sa}
\end{notn}

\begin{ack}
We wish to thank Stefan Schwede for a helpful pointer on symmetric spectra.
\end{ack}

%
%
\sect{Path functors in monoidal categories}
\label{cpfmc}

Higher order homotopy operations in a pointed model category $\eC$, such as
\w[,]{\Topa} \www[,]{\Sa} or \w[,]{\ChR} are usually described in terms of higher
order homotopies, which can be defined in turn in terms of an enrichment of $\eC$
in an appropriate monoidal model category $\eM$ (see, e.g., \cite{BJTurH}).
We here abstract the minimal properties of such an $\eM$ needed for the
construction of higher operations.

\begin{defn}\label{dsqmc}
A \emph{monoidal path category} is a functorially complete and cocomplete
pointed monoidal category \w[,]{\lra{\eM,\otimes,1}} equipped with an \emph{path}
endofunctor \w{P:\eM\to\eM} and natural transformations \w[,]{\rh\sb{X}:PX\to X}
\w[,]{\htl:PX\otimes Y\to P(X\otimes Y)} and
\w[.]{\htr:X\otimes PY\to P(X\otimes Y)}

We require that the following diagrams commute:

\begin{enumerate}
\renewcommand{\labelenumi}{(\alph{enumi})~}
\item Constant path combinations:
\mydiagram[\label{eqpnatt}]{
PX\otimes Y \ar[rr]^{\htl} \ar[d]^{\rh\sb{X}\otimes\Id\sb{Y}} &&
P(X\otimes Y) \ar[d]_{\rh\sb{X\otimes Y}} &&
X\otimes PY \ar[rr]^{\htr} \ar[d]^{\Id\sb{X}\otimes\rh\sb{Y}} &&
P(X\otimes Y) \ar[d]_{\rh\sb{X\otimes Y}} \\
X\otimes Y \ar[rr]^{=} && X\otimes Y && X\otimes Y \ar[rr]^{=}  && X\otimes Y
}

\item Coalgebra structure:
\mydiagram[\label{eqsimpid}]{
P(PX) \ar[rr]^{P(\rh\sb{X})} \ar[d]_{\rh\sb{PX}} && PX \ar[d]^{\rh\sb{X}} \\
PX \ar[rr]^{\rh\sb{X}} && X.
}
\item Left and right constants:
\mydiagram[\label{eqchtheta}]{
PX\otimes PY \ar[rr]^{\htr} \ar[d]_{\htl} &&
P(PX\otimes Y) \ar[d]^{P\htl} \\
P(X\otimes PY) \ar[rr]_{P\htr} && P\sp{2}(X\otimes Y)
}
\item From \wref{eqchtheta} we see that there are natural transformations
$$
\htij{i}{j}~:~P\sp{i}X\otimes P\sp{j}Y~\to~P\sp{i+j}(X\otimes Y)
$$
\noindent for any \w[,]{i,j\geq 0} defined
$$
\htij{i}{j}~:=~P\sp{i+j-1}(\htl)\circ\cdots\circ P\sp{j}(\htl)\circ
P\sp{j-1}(\htr)\circ\cdots\circ \htr~.
$$
\noindent These are required to be associative, in the obvious sense.
\item If we let \w{P\sp{n}X} denote the result of applying the functor
\w{P:\eM\to\eM} to $X$ $n$ times (with \w[),]{P\sp{0}\DEF\Id\sb{\eM}}
we have \w{n+1} different natural transformations
\w{\di{i}\sp{n}:P\sp{n+1}X\to P\sp{n}X} \wb[,]{i=0,\dotsc, n} defined
\begin{myeq}\label{eqsimpface}
\di{i}~=~\di{i}\sp{n}~\DEF~P\sp{i}(\rh\sb{P\sp{n-i}X})~.
\end{myeq}
\noindent The natural transformations \w{\htij{i}{j}} are required to satisfy
the identities:
\begin{myeq}\label{eqsimpmonoi}
\di{k}\sp{n-1}\circ \htij{i}{j}=
\begin{cases}
\htij{i-1}{j}\circ(\di{k}\sp{i-1}\otimes\Id) & \text{if}\hsn 0\leq k<i\\
\htij{i}{j-1}\circ(\Id\otimes\di{k-i}\sp{j-1}) & \text{if}\hsn i\leq k<n
\end{cases}
\end{myeq}
\noindent for every \w[.]{0\leq k<i+j=n}
\end{enumerate}
\end{defn}

\begin{remark}
The commutativity of  \wref{eqsimpid}  implies that the natural
transformations of \wref{eqsimpface} satisfy the usual simplicial identities
\begin{myeq}
\label{eqsimpids}
\di{i}\sp{n-1}\circ \di{j}\sp{n}~=~\di{j-1}\sp{n-1}\circ\di{i}\sp{n}
\end{myeq}
\noindent for all \w[.]{0\leq i<j\leq n}
\end{remark}

\begin{mysubsection}{Paths and cubes}\label{spathcube}
The natural setting where such path categories arise is when a monoidal
category $\eM$ is also \emph{simplicial}, in the sense of \cite[II, \S 1]{QuiH}.
More specifically, we require the existence of an \emph{unpointed path functor}
\w{\uPb{-}:\eM\to\eM} which behaves like a mapping space from the interval
\w[,]{[0,1]}  so we have natural transformations

\begin{enumerate}
\renewcommand{\labelenumi}{(\alph{enumi})~}
\item \w{\es{0},\es{1}:\uP{X}\to X} (evaluation at the two endpoints),
\item \w{s:X\to\uP{X}} with \w{\es{0}s=\es{1}s=\Id} (the constant path),
and
\item \w{\thl:\uP{X}\otimes Y\to\uPb{X\otimes Y}} and
\w{\thr:X\otimes\uP{Y}\to\uPb{X\otimes Y}} (paths in a product).
\end{enumerate}

These make the following diagrams commute:
\mydiagram[\label{eqpmc}]{
\uP{X}\otimes Y \ar[rr]^{\thl} \ar[d]^{\es{i}\sb{X}\otimes\Id\sb{Y}} &&
\uPb{X\otimes Y} \ar[d]_{\es{i}\sb{X\otimes Y}} &
X\otimes\uP{Y} \ar[rr]^{\thr} \ar[d]^{\Id\sb{X}\otimes\es{i}\sb{Y}} &&
\uPb{X\otimes Y} \ar[d]_{\es{i}\sb{X\otimes Y}} \\
X\otimes Y \ar@/^{1.0em}/[u]\sp{s\otimes\Id} \ar[rr]^{=} &&
X\otimes Y \ar@/_{1.0em}/[u]\sb{s} &
X\otimes Y \ar[rr]^{=} \ar@/^{1.0em}/[u]\sp{\Id\otimes s} &&
X\otimes Y \ar@/_{1.0em}/[u]\sb{s}
}
\noindent for \w[,]{i=0,1} as well as
\mydiagram[\label{eqcubid}]{
\uPn{2}{X} \ar[rr]^{\uPb{\es{i}\sb{X}}} \ar[d]_{\es{j}\sb{\uP{X}}} &&
\uP{X} \ar[d]^{\es{j}\sb{X}\hsp\text{and}\hsp} &&
\uP{X}\otimes\uP{Y} \ar[rr]^{\thr} \ar[d]_{\thl} &&
\uPb{\uP{X}\otimes Y} \ar[d]^{\uPb{\thl}}\\
\uP{X} \ar[rr]_{\es{i}\sb{X}} && X &&
\uPb{X\otimes \uP{Y}} \ar[rr]_{\uPb{\thr}} && \uPnb{2}{X\otimes Y}
}
\noindent for \w[.]{i,j\in\{0,1\}}

We may then define the required (pointed) path functor \w{P:\eM\to\eM} by the
functorial pullback diagram:
\mydiagram[\label{eqpath}]{
\ar @{} [drr] |<<<<<{\framebox{\scriptsize{PB}}}
PX \ar[rr] \ar[d] && \uP{X} \ar[d]^{\es{0}} \\
\ast \ar[rr] && X~.
}
The commutativity of the right hand square in \wref{eqcubid}
allows us to define either composite to be the natural transformation
\w[.]{\thij{1}{1}:\uP{X}\otimes\uP{Y}\to\uPnb{2}{X\otimes Y}}

We see that \w{\thl} induces a natural transformation
\w[,]{\htl:PX\otimes Y\to P(X\otimes Y)} and similarly
\w[,]{\htr:X\otimes PY\to P(X\otimes Y)}
making \wref{eqpnatt} commute.

Moreover, from \wref{eqpmc}  we see that
\wref{eqchtheta} commutes, and that the natural transformations
\w{\htij{i}{j}} are associative and satisfy \wref[.]{eqsimpmonoi}
\end{mysubsection}

\begin{example}\label{egtop}
The motivating example is provided by \w[,]{\eM=\Topa} with the
monoidal structure given by the smash product \w[,]{\otimes\DEF\wedge} and
\w{\uP{X}\DEF\mapt(I,X)} the mapping space out of the interval
\w[.]{I\DEF\Delta[1]\sb{+}} Thus \w{PX} is the usual pointed path space.
Here \w{\mapt(X,Y)} denotes the set \w{\Hom\sb{\Topa}(X,Y)} equipped with the
compact-open topology.
\end{example}

\begin{example}\label{egsimpset}
Similarly for \w[,]{\Sa} again with the smash product \w{\otimes\DEF\wedge}
and \w[,]{\uP{X}\DEF\mapa(\Delta[1]\sb{+},X)} where \w{\mapa(X,Y)\in\Sa} denotes
the simplicial mapping space with
\w[.]{\mapa(X,Y)\sb{n}:=\Hom\sb{\Sa}(X\times\Delta[n]\sb{+},Y)}

When $X$ is a Kan complex, we can use Kan's model for
\w[,]{PX} where \w[,]{(PX)_{n}:=\Ker(d_{1}d_{2}\dotsc d_{n+1}:X\sb{n+1}\to X\sb{0})}
and \w{\rh\sb{X}:PX\to X} is \w{d^{i}_{0}} in simplicial dimension $i$.
\end{example}

\begin{example}\label{egspec}
Another variant is provided by a suitable category \w{\Sp} of spectra with
strictly associative smash product $\wedge$, such as the $S$-modules of
\cite{EKMMayR}, the symmetric spectra of \cite{HSSmiS}, and the orthogonal
spectra of \cite{MMSShipM}. One again has function spectra \w[,]{\maps(X,Y)}
which can be used to define \w{\uP{X}} and \w[.]{PX} The unit is the sphere
spectrum \w[.]{S\sp{0}}
\end{example}

\begin{example}\label{egchaincx}
For chain complexes of $R$-modules we have a monoidal structure with the
tensor product
\w[.]{(\bA\otimes\bB)\sb{n}\DEF\bigoplus\sb{i+j=n}\,A\sb{i}\otimes B\sb{j}}

Recall that the \emph{function complex} \w{\hHom(\bA,\bB)} is given by
\begin{myeq}
\label{eqfuncx}
\hHom(\bA,\bB)\sb{n}~\DEF~\prod\sb{i \in \bZ}\ \Hom(A\sb{i}, B\sb{i+n})~,
\end{myeq}
\noindent with \w{\partial\sb{n}((f\sb{i})\sb{i \in \bZ})\DEF
(\partial\sp{B}\sb{i+n} f\sb{i} - (-1)\sp{n}f\sb{i-1}
\partial\sp{A}\sb{i})\sb{i \in \bZ}} for
\w[.]{(f\sb{i}\colon A\sb{i}\to B\sb{i+n})\sb{i\in\bZ}}

Thus for \w{\eM=\ChR} we may set
\w[,]{\uP{X}\DEF\uHom(C\sb{\ast}(\Delta[1];R),X)} and see that \w{P\bA} has
\begin{myeq}\label{eqpathcc}
(PA)\sb{n}=A\sb{n}\oplus A\sb{n+1}\hsm\text{with }\hsm
\partial(a,a')=(\partial a,\partial a'+(-1)\sp{n+1}a)~,
\end{myeq}
\noindent and \w{\rh\sb{\bA}} the projection.
\end{example}

\begin{mysubsection}{Cores and elements}\label{sce}
In any monoidal path category \w{\lra{\eM,\otimes,1,\uPb{-}}}
and for any \w[,]{X\in\eM} we can think of \w{\Hom\sb{\eM}(1,X)} as the
`underlying set' of $X$, and think of a map \w{f:1\to X} in $\eM$ as an
`element' of $X$.

More generally, we may have a suitable monoidal subcategory $\cI$ of $\eM$, which
we call a \emph{core}, and define a \emph{generalized element} of $X$ to be
any map \w{f:\alpha\to X} in $\eM$ with \w[.]{\alpha\in\cI}
\end{mysubsection}

\begin{example}\label{egcore}
We may always choose \w{\cI=\{1\}} to consist of the unit of $\eM$ alone.
However, in some cases other natural choices are possible:

\begin{enumerate}
\renewcommand{\labelenumi}{(\alph{enumi})~}
\item In the three examples of \S \ref{egtop}, \S \ref{egsimpset}, and
\S \ref{egspec}, we can let \w{\cI\sb{S}\DEF\{S\sp{n}\}\sb{n=0}\sp{\infty}}
consist of all (non-negative dimensional) spheres \wh this is evidently closed
under \w[.]{\otimes=\wedge}
\item In the category of chain complexes over a ring $R$ (\S \ref{egchaincx}),
we let \w[,]{\cI\sb{R}\DEF\{\tM{R}{n}\}\sb{n\in\bZ}} where
\w{\tM{R}{n}} is the Moore chain complex with \w{\tMn{R}{n}{i}=R} for
\w[,]{i=n} and $0$ otherwise. Again we see that
\w[,]{\tM{R}{p}\otimes\tM{R}{q}=\tM{R}{p+q}} so \w{\cI\sb{R}} is indeed a monoidal
subcategory of \w[.]{(\ChR, \otimes\sb{R}, \tM{R}{0})}

We see that a generalized element in a chain complex \w{\bA} is now a
map \w{f:\tM{R}{n}\to\bA} in \w{\ChR} \wwh that is, an $n$-cycle in \w[.]{\bA}
\item Other examples are also possible \wh for example, if
\w{\cI':=\{\Mt{\bZ/p}{n}\}\sb{n=1}\sp{\infty}}
is the collection of mod $p$ Moore spaces, representing mod $p$ homotopy
groups (see \cite{NeiP}), then it is not itself a monoidal subcategory of
\w[,]{(\Topa,\wedge,S\sp{0})} since it is not closed under smash products.  However,
when $p$ is odd, the collection of finite wedges of such Moore spaces \emph{is}
monoidal, by \cite[Corollary 6.6]{NeiP}.
\end{enumerate}
\end{example}

%
%
\sect{Higher order chain complexes}
\label{chocc}

The structure defined in the previous section suffices to define higher order
chain complexes, as in \cite{BBlaC}:

\begin{mysubsection}{Categories enriched in monoidal path categories}
\label{scempc}
Let $\eC$ be a category enriched in a monoidal path category
\w[,]{\lra{\eM,\otimes,1,P}} so that for any \w{a,b\in\Obj\eC} we have a
\emph{mapping object} \w{\MapC(a,b)} in $\eM$, and for any \w{a,b,c\in\Obj\eC}
we have a \emph{composition map}
$$
\mu=\mu\sb{a,b,c}~:~\MapC(b,c)~\otimes~\MapC(a,b)~\longrightarrow~\MapC(a,c)
$$
\noindent (written in the usual order for a composite), satisfying the standard
associativity rules.

As in \S \ref{sce}, we can think of a morphism \w{f:1\to\MapC(a,b)} in $\eM$ as an
`element' of \w[,]{\MapC(a,b)} or simply a \emph{map} \w[.]{f:a\to b}
In particular, we have `identity maps' \w{\Id\sb{a}} in
\w{\MapC(a,a)} for each \w[,]{a\in\Obj\eC} satisfying the usual unit rules.

In addition, a morphism  \w{F:1\to P\MapC(a,b)} is called a \emph{nullhomotopy}
of \w[.]{f:=\rh\sb{\MapC(a,b)}\circ F} Higher order nullhomotopies are defined by
maps \w[.]{F:1\to P\sp{i}\MapC(a,b)}

The functoriality of $P$ implies that we can also compose
(higher order) nullhomotopies by means of the composite of
\begin{myeq}\label{eqcompnhtpies}
\begin{split}
P^{i}\MapC(b,c)\otimes P\sp{j}\MapC(a,b) &~\xra{\htij{i}{j}}
P\sp{i+j}[\MapC(b,c)\otimes \MapC(a,b)]\\
&\xra{P\sp{i+j}{\mu}} P\sp{i+j}{\MapC(a,c)}~,
\end{split}
\end{myeq}
\noindent which we denote by
\w[.]{\hmu{i,j}:P\sp{i}\MapC(b,c)\otimes P\sp{j}\MapC(a,b)\to P\sp{i+j}{\MapC(a,c)}}
Again, the maps \w{\hmu{(-,-)}} are associative.

For a general core \w{\cI\subseteq\eM} (cf.\ \S \ref{sce}), we have generalized
elements given by maps \w{f:\alpha\to\MapC(a,b)} for \w[.]{\alpha\in\cI}
We use the fact that $\cI$ is a monoidal subcategory to define the
composite of \w{f:\alpha\to\MapC(a,b)} with \w{g:\beta\to\MapC(b,c)}
\wb{\beta\in\cI} to be the composite in $\eM$ of
\begin{myeq}\label{eqcompgenm}
\beta\otimes\alpha~\xra{g\otimes f}~\MapC(b,c)\otimes\MapC(a,b)~\xra{\mu}~
\MapC(a,c)~,
\end{myeq}
\noindent and similarly for generalized (higher order) nullhomotopies.

From \wref{eqsimpmonoi} we see that:
\begin{myeq}\label{eqsimpcomp}
\di{k}\sp{n-1}\circ \hmu{i,j}=
\begin{cases}
\hmu{i-1,j}\circ(\di{k}\sp{i-1}\otimes\Id) & \text{if}\hsn 0\leq k<i\\
\hmu{i,j-1}\circ(\Id\otimes\di{k-i}\sp{j-1}) & \text{if}\hsn i\leq k<i+j
\end{cases}
\end{myeq}
\noindent for every \w[.]{0\leq k<i+j=n}
\end{mysubsection}

\begin{remark}\label{rpathcube}
If the path structure $P$ comes from a unpointed path structure \w{\uPb{-}} as
in \S \ref{spathcube}, a morphism \w{F:1\to \uP{\MapC(a,b)}} in $\eM$ is called a
\emph{homotopy} \w{F:f\sb{0}\sim f\sb{1}} between
\w{f\sb{0}\DEF\es{0}\sb{\MapC}\circ F} and
\w[.]{f\sb{1}\DEF\es{1}\sb{\MapC}\circ F}

Higher order homotopies are defined by maps \w[,]{F:1\to\uPn{i}\MapC(a,b)} and
the functoriality of \w{\uPb{-}} implies that we can compose (higher order)
homotopies by means of the composite of
$$
\xymatrix@R=15pt@C=15pt{
\uPn{i}{\MapC(b,c)}\otimes\uPn{j}{\MapC(a,b)} \ar[r]^{\thij{i}{j}} &
\uPn{i+j}{[\MapC(b,c)\otimes\MapC(a,b)]} \ar[r]^<<<<<{\uPn{i+j}{\mu}} &
\uPn{i+j}{\MapC(a,c)}~,
}
$$
\noindent which we denote by
\w[.]{\tmu{i,j}:\uPn{i}{\MapC(b,c)}\otimes\uPn{j}{\MapC(a,b)}\to
\uPnb{i+j}{\MapC(a,c)}}
These induce the maps \w[,]{\hmu{i,j}} as in \S \ref{spathcube}.
\end{remark}

\begin{defn}\label{dhocc}
Assume given a monoidal path category \w{\lra{\eM,\otimes,1,P}} with core
$\cI$ in $\eM$ (cf.\ \S \ref{sce}), and choose an ordered set
\w{\Gamma=(\gamma\sb{1},\dotsc,\gamma\sb{N})} of $N$ core elements.

An \emph{$n$-th order chain complex
\w{\dK=\lra{K,\{\{\Fud{k}{i}\}\sb{i=k+1}\sp{N}\}\sb{k=0}\sp{n}}} over $\eM$
(for $\Gamma$) of length \w{N\geq n+2}} consists of:
\begin{enumerate}
\renewcommand{\labelenumi}{(\alph{enumi})~}
\item A category $K$ enriched over $\eM$, with
\w{\Obj(K)=\{a\sb{0},\dotsc,a\sb{N}\}} and
\begin{myeq}\label{eqobjhchcx}
\MapK(a\sb{i},\,a\sb{j})~=~
\begin{cases}1 \amalg \ast& \text{if}\ i=j\\
\ast & \text{if}\ i<j~.
\end{cases}
\end{myeq}
\noindent $K$ will be called the \emph{underlying category} of the
$n$-th order chain complex $\dK$.
\item For each \w{0\leq k\leq n} and \w[,]{i=k+1,\dotsc N}
generalized elements
$$
\Fud{k}{i}:\gamma\sb{i-k}\otimes\dotsc\otimes\gamma\sb{i}\to
P\sp{k}\MapK(a\sb{i},\,a\sb{i-k-1})
$$
\noindent such that
\begin{myeq}\label{eqhocc}
\di{t}\circ\Fud{k}{i}~=~\hmu{k-t-1,t}(\Fud{k-t-1}{i-t-1}\otimes\Fud{t}{i})
\end{myeq}
\noindent for all \w[.]{0\leq t<k}
\end{enumerate}

When \w[,]{N=n+2} we simply call $\dK$ an \emph{$n$-th order chain complex}.
\end{defn}

\begin{remark}\label{rhocc}
Typically we are given a fixed category $\eC$ enriched in a monoidal path category
\w[,]{\lra{\eM,\otimes,1,P}} and the underlying category $K$ for a higher order
chain complex $\dK$ will simply be a finite subcategory of $\eC$ (usually not
full, because of condition \wref[).]{eqobjhchcx} Such a $\dK$ will be called an
\emph{$n$-th order chain complex in $\eC$}.
\end{remark}

\begin{defn}\label{dmaphocc}
Given an $n$-th order chain complex
\w{\dK=\lra{K,\{\{\Fud{k}{i}\}\sb{i=k+1}\sp{N}\}\sb{k=0}\sp{n}}} over $\eM$
(for $\Gamma$) of length $N$, and an enriched functor \w{\phi:K\to L} over
$\eM$ (which we may assume to be the identity on objects, with $L$ also satisfying
\wref[),]{eqobjhchcx} the \emph{induced} $n$-th order chain complex
\w{\dL=\lra{L,\{\{\Gud{k}{i}\}\sb{i=k+1}\sp{N}\}\sb{k=0}\sp{n}}} over $\eM$
(for the same $\Gamma$) is defined by setting
$$
\Gud{k}{i}:=\phi(\Fud{k}{i}):\gamma\sb{i-k}\otimes\dotsc\otimes\gamma\sb{i}\to
P\sp{k}\MapL(a\sb{i},\,a\sb{i-k-1})
$$
\noindent for all \w{0\leq k\leq n} and \w[.]{k<i\leq N}
\end{defn}

\begin{remark}
\noindent Note that we do \emph{not} assume that we have $n$-th order
nullhomotopies
\w{\Fud{n}{i}\in P\sp{n}\MapK(a\sb{i},\,a\sb{i-n-1})} (for \w[)]{i>n}
satisfying \wref[.]{eqhocc}

However, from \wref{eqhocc} and \wref{eqsimpcomp} we see that:
$$
\di{s}\circ\di{t}\circ\Fud{k}{i}~=~\hmu{k-t-2,t}
(\hmu{k-s-t-2,s}(\Fud{k-s-t-2}{i-s-t-2}\otimes\Fud{s}{i-t-1})\otimes\Fud{t}{i})
$$
\noindent if \w[,]{s+t<k-1} and
$$
\di{s}\circ\di{t}\circ\Fud{k}{i}~=~\hmu{k-t-1,t-1}(\Fud{k-t-1}{i-t-1}\otimes
\hmu{k-s-2,s+t-k+1}(\Fud{k-s-2}{i-s-t+k-2}\otimes\Fud{s+t-k+1}{i}))
$$
\noindent if \w[.]{k-1\leq s+t} Thus from the simplicial identity
\w{\di{s}\circ\di{t}=\di{t-1}\circ\di{s}} for \w{0\leq s<t} we deduce that
the maps \w{\{\Fud{k}{i}\}} must satisfy:
\begin{myeq}\label{eqassociativity}
\hmu{}(\Fud{r}{i-s-t-2}\otimes\Fud{s}{i-t-1}\otimes\Fud{t}{i})
\begin{cases}
\hmu{}(\Fud{r+1}{i-s-t-3}\otimes\Fud{t-1}{i-s-1}\otimes\Fud{s}{i})
& \text{if}\ s< t\\
\hmu{}(\Fud{t}{i-r-s-2}\otimes\Fud{s}{i-r-1}\otimes\Fud{r}{i})
& \text{if}\ s\geq r~\text{and}~ t=0\\
\hmu{}(\Fud{s+1}{i-r-t-3}\otimes\Fud{r}{i-t-2}\otimes\Fud{t-1}{i})
& \text{if}\ s\geq r~\text{and}~ t> 0~,
\end{cases}
\end{myeq}
\noindent where we have simplified the notation using the associativity
of $\mu$.
\end{remark}

\begin{mysubsection}{A cubical description}
\label{scdesc}
Higher order chain complexes were originally defined in \cite[\S 4]{BBlaC} in
terms of a cubical enrichment, which is well suited to describing higher
homotopies. In general, for an \wwb{n-1}st order chain complex
\begin{myeq}\label{eqnmochain}
a\sb{n+1}~\xra{\Fud{0}{n+1}}~a\sb{n}~\xra{\Fud{0}{n}}~
a\sb{n-2}~\to~ \dotsc~\to~a\sb{1}~\xra{\Fud{0}{1}}~a\sb{0}~,
\end{myeq}
\noindent we may describe the choices of higher homotopies \w{\Fud{k}{i}}
succinctly by arranging them as the collection of all the cubical faces in
the boundary of \w{I\sp{n+2}} containing a fixed vertex (which is indexed by
\w[).]{\Fud{0}{1}\otimes\Fud{0}{2}\otimes\dotsc\Fud{0}{n}\otimes\Fud{0}{n+1}}

The $k$-faces are indexed by
\begin{myeq}\label{eqkfaces}
\Fud{k\sb{1}}{i\sb{1}}\otimes\dotsc
\otimes\Fud{k\sb{r}}{i\sb{r}}~\in~
P\sp{k\sb{1}}\MapK(a\sb{i\sb{1}},\,a\sb{0})\otimes\dotsc\otimes
P\sp{k\sb{r}}\MapK(a\sb{n+1},\,a\sb{n-k\sb{r}})~,
\end{myeq}
\noindent with \w[,]{\sum\sb{j=1}\sp{r}\,k\sb{j}=k}
\w[,]{i\sb{j}=\sum\sb{t=1}\sp{j-1}\,(k\sb{t}+1)} and \w{r=n-k+1} (so
\w{i\sb{1}=k\sb{1}+1} and \w[).]{i\sb{r}=n+1}

By intersecting the corner of \w{\partial I\sp{n+2}}
with a transverse hyperplane in \w{\bR\sp{n+1}} we obtain an \wwb{n+1}simplex
$\sigma$, whose $n$-faces correspond to the \wwb{n+1}facets of the corner,
and so on.
More precisely, the cone on this simplex (with cone point the chosen vertex $v$ of
\w[)]{I\sp{n+2}} is homeomorphic to \w[,]{I\sp{n+2}} with each \wwb{n+1}face of
the cone obtained from an \wwb{n+1}facet $\tau$  of the corner by identifying the
$n$-corner opposite $v$ in $\tau$ to a single $n$-simplex in the base of the cone.
See Figure \ref{fig2}.

\setcounter{figure}{\value{thm}}\stepcounter{subsection}
\begin{figure}[htbp]
%
%
\begin{center}
\begin{picture}(140,130)(0,-10)
%
%
\put(0,0){\circle*{3}}
\multiput(1,0)(3,0){25}{\circle*{.5}}
\put(80,0){\circle*{3}}
\multiput(0,1)(0,3){25}{\circle*{.5}}
\put(0,80){\circle*{3}}
\put(0,80){\line(1,0){80}}
\put(79,79){\circle*{5}}
\put(80,78){\line(0,-1){78}}
%
%
\multiput(0,80)(3,2){20}{\circle*{.5}}
\put(60,120){\circle*{3}}
\put(140,120){\circle*{3}}
\put(80,80){\line(3,2){60}}
\multiput(60,120)(3,0){26}{\circle*{.5}}
%
%
\multiput(80,0)(3,2){20}{\circle*{.5}}
\put(140,40){\circle*{3}}
\multiput(140,40)(0,3){27}{\circle*{.5}}
\thicklines
\put(40,80){\line(1,-1){40}}
\bezier{300}(40,80)(75,90)(110,100)
\put(80,40){\line(1,2){30}}
\end{picture}
\end{center}
\caption{\label{fig2}Corner of $3$-cube and transverse $2$-simplex}
\end{figure}
\setcounter{thm}{\value{figure}}

This explains why the maps \w{\di{i}\sp{n}:P\sp{n+1}X\to P\sp{n}X}
of \S \ref{dsqmc}, which relate the various $\otimes$-composites appearing as
facets of \w[,]{\partial I\sp{n+1}} satisfy simplicial, rather than cubical,
identities.
\end{mysubsection}

\begin{example}\label{egcube}
Consider a second order chain \vsm complex

\mydiagram[\label{eqtodathree}]{
&&& \ast &&&&&\\
&& \ast & &&& \ast &&\\
a \ar[rr]^{k} \ar@/^{2.8pc}/[rrrr] \ar@/^{6.7pc}/|(.89){\hole}[rrrrrr] &&
b \ar[rr]^{h} \ar@{=>}[u]_{h\circ k} \ar@/_{7.0pc}/[rrrrrr]
\ar@/_{2.8pc}/[rrrr] &&
c \ar@{=>}[uu]^{g\circ h\circ k} \ar[rr]^{g} \ar@/^{2.8pc}/[rrrr]
\ar@{=>}[d]_{g\circ h} &&
d \ar@{=>}[dd]_{f\circ g\circ h} \ar@{=>}[u]_{f\circ g} \ar[rr]^{f} && e\\
 &&&& \ast &&&&\\
 &&&&& \vsp\ast &&&
}
\vsn\quad

\noindent in \w[,]{\Topa} say, in which we have \w{n+1=4} composable maps: \
\w[,]{\Fud{0}{1}=f} \w[,]{\Fud{0}{2}=g} and so on, with all adjacent composites
nullhomotopic.

In this case we may choose nullhomotopies as indicated, namely:
\w{\Fud{1}{2}=f\circ g} in \w{\mapa(c,e)\sp{I\sp{1}}} (with
\w{\es{0}(f\circ g)=\ast} and \w[),]{\es{1}(f\circ g)=fg}
\w{\Fud{1}{3}=g\circ h} in \w[,]{\mapa(b,d)\sp{I\sp{1}}}
and \w{\Fud{1}{4}=h\circ k} in \w{\mapa(a,c)\sp{I\sp{1}}} \wwh
so that in fact \w{f\circ g} is in the pointed path space \w[.]{P\mapa(c,e)}
Similarly, \w{\Fud{2}{4}=f\circ g\circ h} is a homotopy of nullhomotopies
between \w{h\sp{\ast}(f\circ g)} and \w[.]{f\sp{\ast}(g\circ h)}

The more suggestive notation \w[,]{f\circ g} and so on, is motivated by
the cubical Boardman-Vogt $W$-construction of \cite[\S 3]{BVogHI}, as explained in
\cite[\S 5]{BBlaC}: we think a $k$-th order homotopy as a $k$-cube in the
appropriate mapping spaces.

If we apply the usual composition map
$$
\mu:~\mapa(c,d)~\otimes~\mapa(a,c)\sp{I\sp{1}}~\to~\mapa(a,d)\sp{I\sp{1}}
$$
\noindent to \w[,]{g\otimes h\circ k} we obtain a nullhomotopy of \w[,]{ghk}
and similarly for \w{g\circ h\otimes k} in
\w[.]{\mapa(b,d)\sp{I\sp{1}}~\otimes~\mapa(a,b)} Thus we may ask if these two
nullhomotopies are themselves homotopic (relative to \w[):]{ghk} if so, we have
a $2$-cube \w{g\circ h\circ k} in \w[,]{\mapa(a,d)\sp{I\sp{2}}} which in fact lies
in \w[.]{P\sp{2}\mapa(a,d)} The ``formal'' post-composition with
\w{f\in\mapa(d,e)} yields \w{f\otimes g\circ h\circ k} in
\w[.]{\mapa(d,e)\otimes P\sp{2}\mapa(a,d)} Together with the other two formal
composites \w{f\circ g\circ h\otimes k} in \w{P\sp{2}\mapa(b,e)\otimes\mapa(a,b)}
and \w{f\circ g\otimes h\circ k} in \w[,]{P\mapa(c,e)\otimes P\mapa(a,c)} it
fits into the corner of the $3$-cube described in Figure \ref{fig1} (where we use
both notations \w[,]{\Fud{2}{2}=f\circ g} and so on,
to label facets).

\setcounter{figure}{\value{thm}}\stepcounter{subsection}
\begin{figure}[htbp]
%
%
\begin{center}
\begin{picture}(380,180)(-30,0)
%
%
\put(40,75){\circle*{3}}
\multiput(40,74)(0,-3){23}{\circle*{.5}}
\put(-12,65){\scriptsize{$\Fud{0}{1}\otimes\Fud{0}{2}\otimes \Fud{1}{4}
=f\otimes g\otimes h\circ k$}}
\put(140,75){\circle*{5}}
\put(140,75){\line(-1,0){100}}
\put(143,67){\scriptsize{$f\otimes g\otimes h\otimes k$}}
\put(140,70){\line(0,-1){60}}
\put(154,40){\scriptsize{$f\otimes g\circ h\otimes k\hspace*{3mm}
=\Fud{0}{1}\otimes\Fud{1}{3}\otimes \Fud{0}{4}$}}
\put(154,36){\vector(-2,-1){13}}
\put(140,5){\circle*{3}}
\multiput(42,5)(3,0){33}{\circle*{.5}}
\put(40,5){\circle*{3}}
\put(65,40){\framebox{{\scriptsize{$f\otimes g\circ h\circ k$}}}}
\put(92,29){\scriptsize{$=$}}
\put(93,20){\oval(46,17)}
\put(72,17){\scriptsize{$\Fud{0}{1}\otimes \Fud{2}{4}$}}
%
%
\put(225,169){\circle*{3}}
\multiput(225,169)(-2,-1){93}{\circle*{.5}}
\put(325,169){\circle*{3}}
\multiput(325,170)(-3,0){34}{\circle*{.5}}
\put(325,169){\line(-2,-1){183}}
\put(220,153){\scriptsize{$\Fud{1}{2}\otimes\Fud{0}{3}\otimes\Fud{0}{4}$}}
\put(200,139){\scriptsize{$=f\circ g\otimes h\otimes k$}}
\put(229,137){\vector(1,-1){10}}
\put(140,116){{\framebox{\scriptsize{$f\circ g\otimes h\circ k$}}}}
\put(132,106){\scriptsize{$=$}}
\put(125,96){\oval(48,17)}
\put(105,93){\scriptsize{$\Fud{1}{2}\otimes \Fud{1}{4}$}}
%
%
\multiput(325,170)(0,-3){25}{\circle*{.5}}
\put(325,99){\circle*{3}}
\multiput(325,99)(-2,-1){93}{\circle*{.5}}
\put(245,107){\framebox{{\scriptsize{$f\circ g\circ h\otimes k$}}}}
\put(245,95){\scriptsize{$=$}}
\put(235,84){\oval(48,17)}
\put(215,82){\scriptsize{$\Fud{2}{3}\otimes \Fud{0}{4}$}}
\end{picture}
\end{center}
\caption{\label{fig1}The cubical corner}
\end{figure}
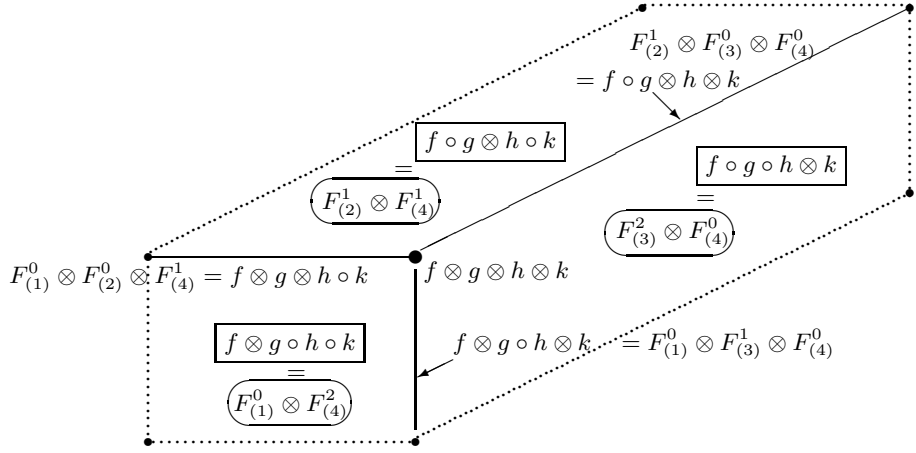
\setcounter{thm}{\value{figure}}

All vertices but the central one represent the zero map, and the dotted edges
represent the trivial nullhomotopy of the zero map (and similarly for the invisible
facets of the cube, representing the trivial second-order homotopy of the trivial
nullhomotopy).
\end{example}

\begin{remark}\label{rcubepath}
The cubical formalism may be used to describe the iterated path complex
\w{P\sp{n}\bA} in the category of chain complexes (see \S \ref{egchaincx}):

We may use the conventions of \S \ref{scdesc} to identify
the $k$-faces of the corner of an $n$-cube \w{I\sp{n}}
(adjacent to a fixed vertex $v$), for \w[,]{0\leq k\leq n} with the
\wwb{k-1}dimensional faces \w{\sud{k}{i}} of the standard
\wwb{n-1}simplex \w{\Delta[n-1]} for \w{0\leq i\leq\binom{n}{k}-1}
(see Figure \ref{fig2}).
Thus \w{I\sp{n}} itself is labelled \w{\sud{n}{0}} (corresponding to
\w[),]{\Delta[n-1]} with the $n$ \wwb{n-1}facets of \w{I\sp{n}} adjacent to $v$
labelled \w[,]{\sud{n-1}{0}=d\sb{0}\sud{n}{0}}
\w[,]{\sud{n-1}{1}=d\sb{1}\sud{n}{0}} and so on. The vertex $v$ is
labelled \w{\sud{0}{0}} (not corresponding to any real face of
\w[).]{\Delta[n-1]}

Then
\begin{myeq}\label{eqnpath}
(P\sp{n}A)\sb{j}~=~\bigoplus\sb{0\leq k\leq n}\
\bigoplus\sb{0\leq i<\binom{n}{k}}\ A\sb{j+k}\upb{\sud{k}{i}}~,
\end{myeq}
\noindent with the differential
\w{\partial\sp{P\sp{n}A}\colon(P\sp{n}A)\sb{j}\to(P\sp{n}A)\sb{j-1}} sending
\w{a\in A\sb{j+k}\upb{\sud{k}{i}}} to \w{\partial\sp{A}(a)} in the summand
\w{A\sb{j+k-1}\upb{\sud{k}{i}}} of \w[,]{(P\sp{n}A)\sb{j-1}} and to
\w{(-1)\sp{n+k+t}a} in the summand \w[.]{A\sb{j+k-1}\upb{d\sb{t}\sud{k}{i}}}

The structure maps \w{\di{i}\sp{n}:P\sp{n}\bA\to P\sp{n-1}\bA} are
given by the projections onto the summands labelled by the $i$-th simplicial
facet of \w{\Delta[n]} and its simplicial faces, for \w[.]{0\leq i\leq n-1}
\end{remark}

\begin{example}\label{egdcubepath}
The double path complex \w{P\sp{2}\bA} is given by
\begin{myeq}\label{eqdpathcc}
(P\sp{2}A)\sb{j}=A\sb{j}\oplus A\sb{j+1}\oplus A\sb{j+1}\oplus A\sb{j+2}~,
\end{myeq}
\noindent with
\begin{myeq}\label{eqdpathccbou}
\partial(x,a,a',y)=(\partial x,\partial a+(-1)\sp{j+1}x,
\partial a'+(-1)\sp{j+1}x,\partial y+(-1)\sp{j}(a-a'))~.
\end{myeq}
\end{example}

\begin{example}\label{egtcubepath}
\noindent Similarly, \w{(P\sp{3}A)\sb{j}} is given by
$$
A\sb{j}\upb{\sud{0}{0}} \oplus A\sb{j+1}\upb{\sud{1}{0}}
\oplus A\sb{j+1}\upb{\sud{1}{1}} \oplus A\sb{j+1}\upb{\sud{1}{2}}
\oplus A\sb{j+2}\upb{\sud{2}{0}} \oplus A\sb{j+2}\upb{\sud{2}{1}}
\oplus A\sb{j+2}\upb{\sud{1}{2}}\oplus A\sb{j+3}\upb{\sud{3}{0}}
$$
\noindent and
\begin{equation*}
\begin{split}
\partial&(a,b\sb{0},b\sb{1},b\sb{2},c\sb{0},c\sb{1},c\sb{2},d)=
(\partial a,\partial b\sb{0}-\tau x,\partial b\sb{1}-\tau x,
\partial b\sb{2}-\tau x,\\
&\partial c\sb{0}+\tau(b\sb{1}-b\sb{0}),
\partial c\sb{1}+\tau(b\sb{2}-b\sb{0}),\partial c\sb{2}+\tau(b\sb{2}-b\sb{1}),
\partial d-\tau(c\sb{2}-c\sb{1}+c\sb{1}))
\end{split}
\end{equation*}
\noindent for \w[.]{\tau=(-1)\sp{j}}
\end{example}

%
%
\sect{Higher Toda brackets}
\label{chtb}

We now show how one may define the higher Toda bracket corresponding to a higher
order chain complex.  First, we need to define the object housing it:

\begin{defn}\label{dloop}
In any monoidal path category \w{\lra{\eM,\otimes,1,\uPb{-}}}
we define the (modified) \emph{$n$-fold loop functor}
\w{\lop{n}:\eM\to\eM} to be the limit:
\begin{myeq}\label{eqloop}
\lop{n}X~:=~\lim\sb{1\leq k\leq n}\ P\sp{k}X
\end{myeq}
\noindent where the limit is taken all the natural maps
\w{\di{i}\sp{k}:P\sp{k}X\to P\sp{k-1}X} of \S \ref{dsqmc}. By \S \ref{scdesc},
we may think of this as a diagram indexed by the dual of the standard
$n$-simplex.

The simplicial identities \wref{eqsimpids} imply that there is a natural
map
\begin{myeq}\label{eqloopmap}
\wrh{n}{X}~:~P\sp{n+1}X~\to~\lop{n}X~,
\end{myeq}
\noindent which composes with the structure maps \w{\pi\sb{t}:\lop{n}X\to P\sp{n}X}
for the limit to yield the face maps \w{\di{t}:P\sp{n+1}X\to P\sp{n}X}
\wb[,]{i=0,\dotsc,n} since \w{\lop{n}X} is the $n$-th matching object for the
restricted augmented simplicial object \w{P\sp{\bullet}X}
(cf.\ \cite[\S 16.3.7]{PHirM}).

For \w{n=0} we set \w[.]{\lop{0}X:=X}
\end{defn}

\begin{example}\label{egloop}
By \S \ref{scdesc}, we may think of \wref{eqloop} as the limit of a diagram
indexed by the dual of the standard $n$-simplex.  Thus \w{\lop{1}X} is the
pullback in:
\mydiagram[\label{eqlopo}]{
\ar @{} [drr] |<<<<<{\framebox{\scriptsize{PB}}}
\lop{1}X \ar[rr] \ar[d] && PX \ar[d]^{\rh\sb{X}} \\
PX \ar[rr]^{\rh\sb{X}} && X,
}
\noindent indexed by the inclusion of the two vertices into \w[,]{\Delta[1]}
while \w{\lop{2}X} is the limit of the diagram:
\begin{myeq}\label{eqlopt}
\vcenter{\begin{xy}
 (0,0)*+{X}="0";
  (0,20)*+{PX}="11";
  (-40,20)*+{PX}="10";
  (40,20)*+{PX}="12";
  (0,40)*+{P\sp{2}X}="21";
  (-40,40)*+{P\sp{2}X}="20";
  (40,40)*+{P\sp{2}X}="22";
  {\ar_{\di{0}\sp{1}=\rh\sb{X}}"10";"0"};
  {\ar^{\di{0}\sp{1}=\rh\sb{X}}"11";"0"};
  {\ar^{\di{0}\sp{1}=\rh\sb{X}}"12";"0"};
  {\ar_{\di{0}\sp{2}}"20";"10"};
  {\ar^(.7){\di{1}\sp{2}}"20";"11"};
  {\ar^{\di{0}\sp{2}}"22";"12"};
  {\ar_(.7){\di{1}\sp{2}}"22";"11"};
  {\ar|(.5){\hole}^(.65){\di{0}\sp{2}}"21";"10"};
  {\ar|(.5){\hole}_(.7){\di{1}\sp{2}}"21";"12"};
\end{xy}}
\end{myeq}
\end{example}

\begin{defn}\label{dhtb}
Let $\dK$ be an \wwb{n-1}st order chain complex (of length \w[)]{n+1}
enriched in a monoidal path category
\w{\lra{\eM,\otimes,1,\uPb{-}}} (for a set
\w{\Gamma=(\gamma\sb{1},\dotsc,\gamma\sb{n+1})} of core elements), as in
\S \ref{dhocc}. If we apply the iterated composition map to each $k$-face of
the form \wref[,]{eqkfaces} we obtain an `element'
\begin{myeq}\label{eqkfacecomp}
\hmu{}(\Fud{k\sb{1}}{i\sb{1}}\otimes\dotsc
\otimes\Fud{k\sb{r}}{i\sb{r}})~:~\gamma\sb{1}\otimes\dotsc\otimes\gamma\sb{n+1}~
\to~P\sp{k}\MapK(a\sb{n+1},\,a\sb{0})
\end{myeq}
\noindent (using the associativity of $\hmu{}$),

From \wref{eqhocc} and \wref{eqsimpcomp} we see that these elements
\wref{eqkfacecomp} are compatible under the face maps
\w[,]{\di{t}:P\sp{k}\MapK(a\sb{n+1},\,a\sb{0})\to
P\sp{k-1}\MapK(a\sb{n+1},\,a\sb{0})}
so that they fit together to define an element
\begin{myeq}\label{eqvalue}
\lra{\dK}~:~\gamma\sb{1}\otimes\dotsc\otimes\gamma\sb{n+1}~\to~
\lop{n-1}\MapK(a\sb{n+1},\,a\sb{0})
\end{myeq}
\noindent which we call the \emph{value of the $n$-th order Toda bracket}
associated to the chain complex $\dK$.

If \w{\lra{\dK}} lifts along the map
\w{\wrh{n-1}{X}:P\sp{n}X\to\lop{n-1}X} of \wref[,]{eqloopmap}
we say that this value of the Toda bracket \emph{vanishes}.
\end{defn}

\begin{remark}\label{rhtb}
Given an \wwb{n-1}st order chain complex
\w{\dK=\lra{K,\{\{\Fud{k}{i}\}\sb{i=k+1}\sp{n+1}\}\sb{k=0}\sp{n-1}}} over $\eM$
(for $\Gamma$), any enriched functor \w{\phi:K\to L} over $\eM$ as in
\S \ref{dmaphocc} takes \w{\lra{\dK}} to
$$
\lra{\dL}~:~\gamma\sb{1}\otimes\dotsc\otimes\gamma\sb{n+1}~\to~
\lop{n-1}\MapL(a\sb{n+1},\,a\sb{0})
$$
\noindent where $\dL$ is the \wwb{n-1}st order chain complex induced by $\phi$,
by functoriality of the limits in $\eM$.
\end{remark}

\begin{mysubsection}{Massey products}
\label{smassey}
Massey products (and their higher order versions) also fit into our setting,
although they cannot be defined as ordinary Toda brackets in a model category.
This is because a (unital associative) differential graded algebra \w{\bA}
over a commutative ground ring $R$ can be thought of as a category $\eC$ with a
single object $\xi$ enriched in \w[,]{(\ChR, \otimes\sb{R}, \tM{R}{0})}
with \w[.]{\Hom_{\eC}(\xi,\xi):=\bA}

In this context we choose the core  of \w{\ChR} to be \w{\cI\sb{R}} as in
\S \ref{egcore}(b). Thus an \wwb{n-1}st order chain complex in \w{\bA} consists of:
\begin{enumerate}
\renewcommand{\labelenumi}{(\alph{enumi})~}
\item The sequence of objects \wh necessarily \w{a\sb{i}=\xi} for all $i$.
\item A sequence of generalized maps
\w{\Fud{0}{i}:\tM{R}{m\sb{i}}\to\Hom_{\eC}(\xi,\xi)}
for \w[,]{i=1,\dotsc,n+1} which may be identified
with an \ww{m\sb{i}}-cycle \w{H\sp{0}\sb{i}\in Z\sb{m\sb{i}}\bA}
(see \S \ref{egcore}(b)).
\item A sequence of generalized nullhomotopies \w{\Fud{1}{i}\in P\bA}
\wb[,]{i=2,\dotsc,n+1} with
\w[.]{\rh\sb{\bA}(\Fud{1}{i})=\mu(f\sb{i-1}\otimes f\sb{i})} From the
description in \S \ref{egchaincx} we see that \w{\Fud{1}{i}} is completely
determined by an element \w{H\sp{1}\sb{i}\in A\sb{m\sb{i}+m\sb{i-1}+1}} with
\w{d(H\sp{1}\sb{i})=H\sp{0}\sb{i-1}\cdot H\sp{0}\sb{i}} (where $d$ is the
differential and $\cdot$ is the multiplication in \w[).]{\bA}
\item From \S \ref{egdcubepath}  we see that a `second-order nullhomotopy'
\w{\Fud{2}{i}\in P\sp{2}\bA} \wb[,]{i=3,\dotsc,n+1} which is
a \wwb{j+2}cycle for \w[,]{j\DEF m\sb{i}+m\sb{i-1}+m\sb{i-2}} is determined
uniquely by the element \w{H\sp{2}\sb{i}\in A\sb{j}} (the last summand in
\wref[).]{eqdpathcc} From the last term in \wref{eqdpathccbou} we see that
\w{\Fud{2}{i}} being a cycle means that
$$
d(H\sp{2}\sb{i})~=~(-1)\sp{j+1}~
(H\sp{0}\sb{i-2}\cdot H\sp{1}\sb{i}-H\sp{1}\sb{i-1}\cdot H\sp{0}\sb{i})~.
$$
\item In general, for each \w{1\leq k<n} and \w[,]{i=k+1,\dotsc n+1}
we have a (generalized) \w{\Fud{k}{i}\in P\sp{k}\bA} which is a
\wwb{j+k}cycle for \w[,]{j\DEF \sum\sb{t=i-k}\sp{i}\,m\sb{t}} with
\begin{myeq}\label{eqhoccp}
\di{t}\circ\Fud{k}{i}~=~\Fud{k-t-1}{i-t-1}\cdot\Fud{t}{i}~,
\end{myeq}
\noindent and from the description in \S \ref{rcubepath} we see that again
\w{\Fud{k}{i}} is completely determined by the component \w{H\sp{k}\sb{i}}
in the summand \w[,]{A\sb{j+k}\upb{\sud{k}{i}}}
with
$$
d(H\sp{k}\sb{i})~=~(-1)\sp{k+j+1}\,\sum\sb{t=0}\sp{k-1}\
(-1)\sp{t}H\sp{t}\sb{i-k+t}\cdot H\sp{k-t-1}\sb{i}~.
$$
\end{enumerate}

Thus by Definition \ref{dhtb} we see that the value of the
\wwb{n+1}st order Toda bracket associated to this \wwb{n-1}st order chain complex
in \w{\bA} is the element in
\w{\lop{n-1}\bA=\lim\sb{1\leq k<n}\ P\sp{k}\bA} determined by the
coherent choice of elements
\begin{myeq}\label{eqcohelts}
H\sp{t}\sb{i-k+t}\cdot H\sp{n-t}\sb{i}~\in A\sb{j+n}
\hsp \text{for}\hsm t=1,\dotsc, n~,
\end{myeq}
\noindent where \w[.]{j\DEF \sum\sb{t=1}\sp{n}\,m\sb{t}}
\end{mysubsection}

%
%
\sect{Higher Toda brackets in model categories}
\label{chtbmc}

In order to \emph{define} the values of higher Toda brackets, all we need is
a category enriched in a monoidal path category $\eM$. However, in applications
we want to use such Toda brackets, either as obstructions to rectifying
diagrams, or as invariants used in computations (e.g., of differentials in
spectral sequence). For this we need to make an additional

\begin{defn}\label{dsmmc}
A \emph{path model category} is a pointed monoidal model category
\w{\lra{\eM,\otimes,1}} in the sense of \cite[Ch.\ 4]{HovM} which satisfies
the conditions of either of \cite[Theorem 1.9,Theorem 1.10]{BMoerH}, and
which is also a simplicial model category as in \cite[II, \S 2]{QuiH},
equipped with a core $\cI$ (cf.\ \S \ref{sce}) consisting of cofibrant objects,
and a natural transformation
\begin{myeq}\label{eqzeta}
\zeta\sb{X,Y,K}:X\sp{K}\otimes Y\sp{K}\to (X\otimes Y)\sp{K}
\end{myeq}
\noindent (natural in \w{X,Y\in\eM} and \w[).]{K\in\eS}
\end{defn}

\begin{remark}\label{rsmmc}
By \cite[Proposition 4.2.19]{HovM}, a path model category actually has a
\ww{\Sa}-model category structure \wh that is, we have functors
\w{(-)\sp{K}:\eM\to\eM} and \w{(-)\otimes K:\eM\to\eM} for every \emph{pointed}
simplicial set \w[,]{K\in\Sa} satisfying the usual axioms.
\end{remark}

\begin{examples}\label{egsmmc}
In practice we shall be interested only in the following examples:
\begin{enumerate}
\renewcommand{\labelenumi}{(\alph{enumi})~}
\item The monoidal structure on \w{\Top} is cartesian, so we actually have
a natural homeomorphism
\w[.]{\tilde{\zeta}:X\sp{K}\times Y\sp{K}\xra{\cong}(X\times Y)\sp{K}}
It is readily verified that in the pointed version \w{\lra{\Topa,\wedge,S^{0}}}
of \S \ref{egtop}, the map $\tilde{\zeta}$ induces
\w[.]{\zeta:X\sp{K}\wedge Y\sp{K}\to(X\wedge Y)\sp{K}}
\item The monoidal structure on $\eS$ is also cartesian, so in
the pointed version \w{\lra{\Sa,\wedge,S^{0}}} of \S \ref{egsimpset}
we also have an induced map as in \wref[.]{eqzeta}
\item If we use symmetric spectra as our model for \w{\Sp} (cf.\ \S \ref{egspec})
we see that the spectrum \w{X\sp{K}} is defined levelwise, so we have \wref{eqzeta}
as for \w[.]{\Topa}
\item In the category \w{\lra{\ChR,\otimes,\tM{R}{0}}} of chain complexes of
$R$-modules (\S \ref{egchaincx}), the monoidal structure is not cartesian,
but the simplicial structure is defined by setting
\w{\bA\sp{K}:=\hHom(C\sb{\ast}K,\bA)} (where \w{C\sb{\ast}K} is the simplicial
chain complex of \w[).]{K\in\eS} The natural transformation \wref{eqzeta} is
induced by the diagonal \w{\Delta:K\to K\times K} in $\eS$.
\end{enumerate}
Note that all of these satisfy the hypotheses of one of
\cite[Theorem 1.9,Theorem 1.10]{BMoerH}, by \cite[\S 1.8]{BMoerH}
and \cite[Proposition A.3.2.4-A.3.2.24]{LuriH}, so they are in fact path
model categories.
\end{examples}

\begin{remark}\label{rloop}
In this case the simplicial structure defines the functor
\w[,]{\uPb{-}:\eM\to\eM} with \w{\uP{X}\DEF X\sp{\Delta[1]}}
(cf.\ \cite[II, \S 1]{QuiH}), and \w{PX\hra X\sp{\Delta[1]}} is defined by the
pullback \wref[.]{eqpath} We can therefore identify \w{P\sp{k}X} for each
\w{k\geq 0} with the subobject of \w{X\sp{[0,1]\sp{k}}} consisting of all maps
of the $k$-cube sending the corner opposite a fixed vertex to the basepoint
(see Figure \ref{fig1}).

Thus \w{\lop{n}X} is a subobject of \w[,]{\lim\sb{k}\ \mapa([0,1]\sp{k},X)}
which by adjunction may be identified with
\w[.]{X\sp{\colim\sb{k}\,[0,1]\sp{k}}} Thus \w{\lop{n}X}
itself is just \w[,]{\mapa(\colim\,\widetilde{[0,1]\sp{k}},\,X)} where the
colimit is now taken over all proper faces of \w[,]{[0,1]\sp{n+1}} and we identify
the corner opposite our chosen vertex of \w{[0,1]\sp{n+1}} to a point.
This colimit is homeomorphic to an $n$-sphere, so \w{\lop{n}X} is homotopy
equivalent to the $n$-fold loop space \w[,]{\Omega\sp{n}X} defined as usual
by iterating the functor \w{\Omega:\eM\to\eM} given by the pullback
\mydiagram[\label{eqoldloop}]{
\ar @{} [drr] |<<<<<{\framebox{\scriptsize{PB}}}
\Omega X \ar[rr] \ar[d] && PX \ar[d]^{\rh\sb{X}} \\
\ast \ar[rr] && X.
}
\end{remark}

\begin{remark}\label{rhtpycat}
In any path model category $\eM$, for any fibrant object $X$ we
have an equivalence relation $\sim$ on the set of morphisms \w{\Hom\sb{\eM}(1,X)}
(cf.\ \S \ref{sce}), given by:
$$
f\sim g\hsm\EQUIV\hsm \exists F:1\to \uP{X}\hsm \text{such that}\hsm
\es{0}\circ F=f\hsm \text{and}\hsm \es{1}\circ F=g~.
$$
\noindent We then define the (pointed) \emph{set of components} \w{\pi\sb{0}X}
to be the set of equivalence classes in \w{\Hom\sb{\eM}(1,X)} under $\sim$.

Now let $\eC$ be a category enriched in $\eM$, and assume the mapping objects
\w{\MapC(a,b)} are fibrant (e.g., if all objects in $\eM$ are fibrant,
as in \w[),]{\Topa} If we denote \w{\pi\sb{0}\MapC(a,b)} simply by
\w[,]{\pih{a}{b}} from \S \ref{scempc} we see that $\mu$ induces an associative
composition on \w[,]{[-,-]} so that this serves as the set of morphisms in the
\emph{homotopy category} \w{\ho\eC} of the $\eM$-enriched category $\eC$
(with the same objects as $\eC$).
\end{remark}

\begin{defn}\label{dcorehtpy}
More generally, if $\cI$ is the core of a path model category $\eM$, for any
core element $\gamma$ (which is cofibrant by Definition \ref{dsmmc}) the
simplicial enrichment \w{\mapM} in $\eM$ allows us to identify \w{[\gamma,X]}
with \w{\pi\sb{0}\mapM(\gamma,X)} (see \cite[II, 2.6]{QuiH}).

Thus if $\eC$ is enriched in $\eM$, we may set
$$
[a,b]\sb{\gamma}~:=~\pi\sb{0}\mapM(\gamma,\MapC(a,b))~.
$$
\noindent for any \w{a,b\in\eC} and \w[.]{\gamma\in\cI}

Note that for any \w{\gamma,\delta\in\cI} and \w[,]{i\geq 0} the
bifunctor $\otimes$, the map \w{\zeta\sb{X,Y,\Delta[i]}} of \wref{eqzeta}
for \w{X:=\MapC(b,c)} and \w[,]{Y:=\MapC(a,b)} and the composition
\w{\mu:X\otimes Y\to Z} (for \w[)]{Z:=\MapC(a,c)} induce natural maps of sets
\begin{equation*}
\begin{split}
(\mapM(\gamma,X)\times\mapM(\delta,Y))\sb{i}~=~&
\Hom\sb{\eM}(\gamma,X\sp{\Delta[i]})\times\Hom\sb{\eM}(\delta,Y\sp{\Delta[i]})~
\xra{\otimes\sb{\ast}}\\
\Hom\sb{\eM}(\gamma\otimes\delta,X\sp{\Delta[i]}\otimes Y\sp{\Delta[i]})&~
\xra{\zeta}~\Hom\sb{\eM}(\gamma\otimes\delta,(X\otimes Y)\sp{\Delta[i]})~
\xra{(\mu\sp{\Delta[i]})\sb{\ast}}\\
&\Hom\sb{\eM}(\gamma\otimes\delta,Z\sp{\Delta[i]})~=~
(\mapM(\gamma\otimes\delta,Z))\sb{i}
\end{split}
\end{equation*}
\noindent and thus a composition map
\w{\nu:\mapM(\gamma,X)\times\mapM(\delta,Y)\to\mapM(\gamma\otimes\delta,Z)}
in $\eS$.  Thus induces an associative composition map
\begin{myeq}\label{eqnu}
\nu\sb{\ast}:[b,c]\sb{\gamma}~\times~[a,b]\sb{\delta}~\to~
[a,c]\sb{\gamma\otimes\delta}~.
\end{myeq}
Thus we have an $\cI$-graded category denoted by \w[,]{\ho\sp{\cI}\eC} called the
\emph{$\cI$-homotopy category} of $\eC$.
\end{defn}

\begin{defn}\label{ddke}
Assume given a path model category $\eM$ with core $\cI$.
We say that a category $K$ enriched in $\eM$ is \emph{fibrant} if
\w{\MapK(a,b)} is fibrant in $\eM$ for any \w[.]{a,b\in K} Note that since each
\w{\gamma\in\cI} is cofibrant, this implies that \w{\mapM(\gamma,\MapK(a,b))} is
a fibrant simplicial set, by SM7.

An enriched functor \w{\phi:K\to L} between categories $K$ and $L$ enriched in
$\eM$ is a \emph{Dwyer-Kan equivalence} if
\begin{enumerate}
\renewcommand{\labelenumi}{(\alph{enumi})~}
\item For all \w[,]{a,b\in\eC} \w{\phi:\MapK(a,b)\to\MapL(\phi(a),\phi(b))}
is a weak equivalence in $\eM$.
\item The induced functor \w{\phi\sb{\ast}:\ho\sp{\cI}K\to\ho\sp{\cI}L} is an
equivalence of $\cI$-graded categories.
\end{enumerate}
\noindent See \cite{SShipE}, and compare \cite{BMoerH}.

We say that such a Dwyer-Kan equivalence is a \emph{trivial fibration} if each
\w{\phi:\MapK(a,b)\to\MapL(\phi(a),\phi(b))} is a fibration in $\eM$.
\end{defn}

By Definition \ref {dsmmc} and \cite[Theorem 1.9-1.10]{BMoerH} we have:

\begin{thm}\label{tcanmc}
There is a canonical model category structure on the category
\w{\MCat} of small categories enriched in any path model category $\eM$, in which
the trivial fibrations and fibrant categories are defined object-wise, and
the weak equivalences are the Dwyer-Kan equivalences.
\end{thm}

\begin{defn}\label{dkhtb}
Let $\eM$ be a path model category with core $\cI$, and let
\w{\dK\up{0}=\lra{K,\{\Fud{0}{i}\}\sb{i=1}\sp{n+1}\}}} be a fixed fibrant
$0$-th order chain complex of length \w{n+1} over $\eM$ for
\w[.]{\Gamma\subseteq\cI} We define \w{\eL\sb{\dK\up{0}}} to be the collection
of all possible fibrant \wwb{n-1}st order chain complexes $\dK$
(of length \w[)]{n+1} extending \w[.]{\dK\up{0}}

Each \w{\dK\in\eL\sb{\dK\up{0}}} has a value
\w[,]{\lra{\dK}:\gamma\sb{1}\otimes\dotsc\otimes\gamma\sb{n+1}\to
\lop{n-1}\MapK(a\sb{n+1},\,a\sb{0})} as in \wref[,]{eqvalue}
which we may identify with a $0$-simplex in the corresponding
simplicial mapping space
\begin{myeq}\label{eqvaluezeros}
\lra{\dK}\in\mapM(\gamma\sb{1}\otimes\dotsc\otimes\gamma\sb{n+1},\
\lop{n-1}\MapK(a\sb{n+1},\,a\sb{0}))\sb{0}~.
\end{myeq}
\noindent By Remark \ref{rloop}
\w{\lop{n-1}\MapK(a\sb{n+1},\,a\sb{0})} is weakly equivalent to the
\wwb{n-1}fold loop space on the mapping space
\w{\map\sb{\eC}(a\sb{n+1},\,a\sb{0})} in $\eM$ (cf.\ \cite[I, \S 2]{QuiH}).
Moreover, we have a natural isomorphism
\begin{myeq}\label{eqadjsimp}
\map\sb{\eM}(Y,X\sp{L})~\xra{\simeq}~\map\sb{\eS}(L,\map\sb{\eM}(Y,X))
\end{myeq}
\noindent for any \w{X,Y\in\eM} and \w{L\in\eS} any finite simplicial set,
by \cite[II, \S 1]{QuiH}), so we may identify the path component
\w{[\lra{\dK}]} of this $0$-simplex with the corresponding element in
\begin{equation*}
\begin{split}
\pi\sb{0}\mapM&(\gamma\sb{1}\otimes\dotsc\otimes\gamma\sb{n+1},\
\Omega\sp{n-1}\MapK(a\sb{n+1},\,a\sb{0}))\\
\cong&~
\pi\sb{0}\Omega\sp{n-1}\mapM(\gamma\sb{1}\otimes\dotsc\otimes\gamma\sb{n+1},\
\MapK(a\sb{n+1},\,a\sb{0}))\\
\cong&~\pi\sb{n-1}\mapM(\gamma\sb{1}\otimes\dotsc\otimes\gamma\sb{n+1},\
\MapK(a\sb{n+1},\,a\sb{0}))
\end{split}
\end{equation*}
\noindent We call the set
$$
\llrra{\dK\up{0}}~:=~
\{[\lra{\dK}]\in\pi\sb{n-1}\mapM(\gamma\sb{1}\otimes\dotsc\otimes\gamma\sb{n+1},\
\MapK(a\sb{n+1},\,a\sb{0}))~:\ \dK\in\eL\sb{\dK\up{0}}
$$
\noindent the \emph{$n$-th order Toda bracket for} \w[.]{\dK\up{0}}  We say that it
\emph{vanishes} if \w[.]{0\in\llrra{\dK\up{0}}}

Of course, \w{\llrra{\dK\up{0}}} may be empty (if there are no \wwb{n-1}st order
chain complexes $\dK$ extending \w[).]{\dK\up{0}} It vanishes if and only if there
is an $n$-th order chain complex extending \w[.]{\dK\up{0}}
\end{defn}

\begin{remark}\label{rsusploop}
When $K$ is a higher chain complex in \w{\eC=\eM} in a monoidal path category
enriched over itself (e.g., for \w{\eM=\Topa} or \w[),]{\Sa}
the homotopy class \w{[\lra{K}]} may be thought of as an element in the group
$$
[\susp\sp{n-1}\gamma\sb{1}\otimes\dotsc\otimes\gamma\sb{n+1}\otimes a\sb{n+1},\
a\sb{0}]\sb{\ast}
$$
\noindent Moreover, \w{[\lra{K}]} vanishes if and only if it represents the
zero element in this group.
\end{remark}

\begin{lemma}\label{lloops}
If $\eM$ is a simplicial model category and \w{f:X\to Y} is
a (trivial) fibration between fibrant objects in $\eM$, then the induced maps
\w{P\sp{k}f:P\sp{k}X\to P\sp{k}Y} and \w{\lop{k}f:\lop{k}X\to\lop{k}Y} are
(trivial) fibrations for all \w[.]{k\geq 1}
Furthermore, if \w{f:X\to Y} is a weak equivalence between
fibrant and cofibrant objects in $\eM$, so are
\w{P\sp{k}f:P\sp{k}X\to P\sp{k}Y} and \w[.]{\lop{k}f:\lop{k}X\to\lop{k}Y}
\end{lemma}

\begin{proof}
Using Axiom SM7 for the simplicial model category $\eM$, the natural isomorphism
\wref[,]{eqadjsimp} and SM7 for $\eS$ itself (cf.\ \cite[II, \S 1-3]{QuiH}),
we see that
\begin{enumerate}
\renewcommand{\labelenumi}{(\alph{enumi})~}
\item Any (trivial) cofibration \w{i:K\hra L} in $\eS$ induces a (trivial)
fibration
\w[,]{i\sp{\ast}:X\sp{L}\epic X\sp{K}} as long as \w{X\in\eM} is fibrant.
\item Any (trivial) fibration \w{f:X\to Y} in $\eM$ induces a (trivial) fibration
\w{f\sb{\ast}:X\sp{K}\epic Y\sp{K}} for any (necessarily cofibrant) \w[.]{K\in\eS}
\end{enumerate}

In particular, let \w{C\sp{n}\sb{+}} denote the sub-simplicial set of the cube
boundary \w{\partial I\sp{n}} consisting of all facets adjacent to a fixed
corner $v$ (i.e., the cubical star of $v$ in \w[),]{\partial I\sp{n}} with
\w{\partial C\sp{n}} its boundary (the cubical link of $v$). The cofibration
\w{i:\partial C\sp{n}\sb{+}\hra C\sp{n}\sb{+}} makes
\w{i\sp{\#}:X\sp{C\sp{n}\sb{+}}\to X\sp{\partial C\sp{n}\sb{+}}} a fibration in $\eM$, by (a).

In particular, the pullback square
\mydiagram[\label{eqcornerfib}]{
\ar @{} [drr] |<<<<<{\framebox{\scriptsize{PB}}}
\lop{n-1}X \ar[rr] \ar[d] &&
X\sp{C\sp{n}\sb{+}} \ar@{->>}[d]^(.4){i\sp{\#}}\\
\ast \ar[rr] && X\sp{\partial C\sp{n}\sb{+}}
}
\noindent defining \w{\lop{n-1}X} (see \S \ref{dloop} and compare \S \ref{scdesc})
is a homotopy pullback (see \cite{MathP}).

Thus if \w{f:X\to Y} is a (trivial) fibration in $\eM$, then the induced map
\w{\lop{n-1}f:\lop{n-1}X\to\lop{n-1}Y} is a (trivial) fibration, by (b).

Similarly, if we consider the (pointed) cofibration sequence in \w[:]{\Sa}
$$
S\sp{0}=\{0,\ast\}~\hra~\Delta[1]\sb{+}=[0,1]\cup\{\ast\}~\epic~\Delta[1]=[0,1]
$$
\noindent (with $\ast$ as basepoint in the first two, and $0$ as the basepoint
in the cofiber), we see from the corresponding fibration sequence in $\eM$:
$$
PX=X\sp{\Delta[1]}~\hra~X\sp{I}~\epic~X\sp{S\sp{0}}=X
$$
\noindent that if \w{f:X\to Y} is a (trivial) fibration in $\eM$, so is
\w[,]{Pf:PX\to PY} by (b) again (see \S \ref{rsmmc} above).
\end{proof}

\begin{lemma}\label{lpathloop}
If $X$ is a fibrant object in a simplicial model category $\eM$, then for
each \w{n\geq 0} the map \w{\wrh{n}{X}:P\sp{n+1}X\to\lop{n}X} of \wref{eqloopmap}
 is a fibration.
\end{lemma}

Note that for \w[,]{n=0} \w{\lop{0}X=X} and \w{\wrh{0}{X}} is simply
\w[.]{\rh\sb{X}:PX\to X}

\begin{proof}
If we consider the map of cofibration sequences (pushouts to $\ast$) in $\eS$:
\mydiagram[\label{eqtwocofibs}]{
\partial C\sp{n}\sb{+} \ar@{^{(}->}[rr]  \ar@{^{(}->}[d] &&
C\sp{n}\sb{+} \ar[rr]  \ar@{^{(}->}[d] &&
C\sp{n}\sb{+}/\partial C\sp{n}\sb{+} \ar@{^{(}->}[d]\\
C\sp{n}\sb{-} \ar@{^{(}->}[rr] && I\sp{n}\ar[rr]  &&
I\sp{n}/\partial C\sp{n}\sb{-}
}
\noindent we see that the natural map
\w{C\sp{n}\sb{+}/\partial C\sp{n}\sb{+} \to I\sp{n}/\partial C\sp{n}\sb{-}} is an inclusion
(cofibration) in \w[,]{\Sa} so the natural map  it  induces \wh namely,
 \w{\wrh{n}{X}::P\sp{n+1}X\to\lop{n}X} \wwh is a fibration by (b) above.

For \w{n=0} this follows directly because  \w{\rh\sb{X}} is a pullback in the following diagram:
\mydiagram[\label{eqpathfib}]{
\ar @{} [drr] |<<<<<{\framebox{\scriptsize{PB}}}
PX \ar[rr] \ar[d]_{\rh\sb{X}} && \uP{X} \ar[d]^{\es{0}\top\es{1}} \\
X \ar[rr]^{\Id\top\ast} && X\times X
}
\noindent where \w{\es{0}\top\es{1}} is a fibration since it is induced by
the cofibration \w{\{0,1\}\hra\Delta[1]} in $\eS$.
\end{proof}

\begin{thm}\label{tdke}
Let $\eM$ be a path model category with core $\cI$, and let
\w{\dK\up{0}=\lra{K,\,\{\Fud{0}{i}\}\sb{i=1}\sp{n+1}\}}} and
\w{\dL\up{0}=\lra{L,\,\{\Gud{0}{i}\}\sb{i=1}\sp{n+1}\}}}
be $0$-th order chain complexes of length \w{n+1} over $\eM$ (for the same
\w[)]{\Gamma\subseteq\cI} with $K$ and $L$ fibrant, and let
\w{\phi\up{0}:\dK\up{0}\to\dL\up{0}} be a map of $0$-th order chain complexes
which is a Dwyer-Kan equivalence. Then the resulting equivalence of categories
\w{\phi\sb{\ast}:\ho\sp{\cI}K\to\ho\sp{\cI}L} induces a bijection between
\w{\llrra{\dK\up{0}}} and \w[.]{\llrra{\dL\up{0}}}
\end{thm}

\begin{proof}
We assume for simplicity that $\phi$ is the identity on objects, so we may
identify both \w{\pi\sb{0}\mapM(\gamma,\MapK(a,a'))} and
\w{\pi\sb{0}\mapM(\gamma,\MapL(a,a'))} as \w[.]{[a,a']\sb{\gamma}}
Similarly we may identify the groups \w{\pi\sb{\ast}\mapM(\gamma,\MapK(a,a'))} and
\w[.]{\pi\sb{\ast}\mapM(\gamma,\MapL(a,a'))}

Given an \wwb{n-1}st order chain complex $\dK$ extending \w[,]{\dK\up{0}} $\phi$
induces an \wwb{n-1}st order chain complex $\dL$ extending \w[,]{\dL\up{0}} as in
\S \ref{dmaphocc}, and takes the value
\w{\lra{\dK}\subset
[a\sb{n+1},\,a\sb{0}]\sb{\gamma\sb{1}\otimes\dotsc\otimes\gamma\sb{n+1}}} to
\w[\vsm.]{\lra{\dL}}

\noindent\textbf{(a)} \ First assume that  \w{\phi\up{0}:\dK\up{0}\to\dL\up{0}}
is a trivial fibration.

To show that the above correspondence is a bijection, let $\dL$ be
an \wwb{n-1}st order chain complex extending \w[.]{\dL\up{0}}
We show by induction on \w{k\geq 0} that we have
an $k$-th order chain complex \w{\dK\up{k}} extending \w[,]{\dK\up{0}} where
\w{\phi\sb{\ast}\dK\up{k}} agrees with $\dL$ to $k$-th order (by assumption
this holds  for \w[).]{k=0}

 In the induction step, we have a \wwb{k-1}st order chain complex
\w{\dK\up{k-1}} such that \w{\phi\sb{\ast}\dK\up{k-1}} agrees with $\dL$ to
\wwb{k-1}st order, which we wish to extend to \w[.]{\dK\up{k}}
Thus we have a commuting diagram
\begin{myeqn}[\label{eqfirstpb}]{
\begin{xy}
(-40,40)*+{\gamma\sb{i-k}\otimes\dotsc\otimes \gamma\sb{i}}="g";
(5,40)*+{P\sp{k}\MapK(a\sb{i},\,a\sb{i-k-1})}="pk";
(5,25)*+{Q\sb{i}}="q";
(14,20)*+{PB}*\frm{-};
(60,25)*+{P\sp{k}\MapL(a\sb{i},\,a\sb{i-k-1})}="pl";
(5,5)*+{\lop{k-1}\MapK(a\sb{i},\,a\sb{i-k-1})}="k";
(60,5)*+{\lop{k-1}\MapL(a\sb{i},\,a\sb{i-k-1})}="l";
{\ar@/^3.5pc/^{\Gud{k}{i}}"g";"pl"};
{\ar@/_2pc/_{\alpha\sb{K}}"g";"k"};
{\ar@{-->}_{\psi}"g";"q"};
{\ar@{-->}_{\simeq}^{\xi}"pk";"q"};
{\ar@{->>}^{P\sp{k}\phi}"pk";"pl"};
{\ar@<-5.5ex>@{->>}_(.6){\wrh{k-1}{K}}"pk";"k"};
{\ar@{->>}^(.4){p\sb{2}}_(.4){\simeq}"q";"pl"};
{\ar@{->>}^{p\sb{1}}"q";"k"};
{\ar@{->>}^{\wrh{k-1}{L}}"pl";"l"};
{\ar^{\lop{k-1}\phi}_{\simeq}"k";"l"};
\end{xy}
}
\end{myeqn}
\noindent in which \w{Q\sb{i}} is the pullback as indicated, and
\w{\alpha\sb{K}:\gamma\sb{i-k}\otimes\dotsc\otimes\gamma\sb{i}\to
\lop{k-1}\MapK(a\sb{i},\,a\sb{i-k-1})} into the limit is induced by the maps
\w{\Fud{k-1}{t}} \wb[.]{t=0,\dotsc,k-1}

Here \w{p\sb{2}} is a trivial fibration and \w{p\sb{1}} is a fibration by
base change (using Lemmas \ref{lloops} and \ref{lpathloop}). The maps
\w{\psi:\gamma\sb{i-1}\otimes\gamma\sb{i}\to Q\sb{i}} and
\w{\xi:P\sp{k}\MapK(a\sb{i},a\sb{i-k-1})\to Q\sb{i}}
exist by the universal property, and $\xi$ is a weak equivalence by the $2$ out of
$3$ property. Factor $\xi$ as
$$
P\sp{k}\MapK(a\sb{i},a\sb{i-k-1})~\xra{j}~
\widehat{P\sp{k}\MapK}(a\sb{i},a\sb{i-k-1})~
\xra{\widehat{\xi}}~Q\sb{i}~,
$$
\noindent where $j$ a trivial cofibration and \w{\widehat{\xi}} is a trivial
fibration.
Since \w{\gamma\sb{i-k}\otimes\dotsc
\otimes\gamma\sb{i}\in\cI} is cofibrant, we have a lifting
as indicated in the solid commuting square:
\mydiagram[\label{eqfirstlift}]{
\ast \ar[rr] \ar@{^{(}->}[d] &&
\widehat{P\sp{k}\MapK}(a\sb{i},\,a\sb{i-k-1}) \ar@{->>}[d]^{\widehat{\xi}}_{\simeq}\\
\gamma\sb{i-k}\otimes\dotsc\otimes\gamma\sb{i} \ar[rr]_{\psi}
\ar@{-->}[rru]^{\widehat{\psi}}&& Q\sb{i}
}
\noindent Since $j$ is a trivial cofibration and \w{\wrh{k}{X}} is
a fibration (for \w[)]{X:=\lop{k-1}\MapK(a\sb{i},a\sb{i-k-1})} by
Lemma \ref{lpathloop}, we have a lift $\zeta$ as indicated in:
\mydiagram[\label{eqsecondlift}]{
P\sp{k}\MapK(a\sb{i},a\sb{i-k-1})  \ar@{=}[rr] \ar@{^{(}->}[d]^{j}_{\simeq} &&
P\sp{k}\MapK(a\sb{i},a\sb{i-k-1}) \ar@{->>}[d]^{\wrh{k-1}{K}} \\
\widehat{P\sp{k}\MapK}(a\sb{i},a\sb{i-k-1}) \ar@{->>}[rr]_{\widehat{\sigma}}
\ar@{-->}[rru]^{\zeta}&& \lop{k-1}\MapK(a\sb{i},a\sb{i-k-1}),
}
\noindent for \w[.]{\widehat{\sigma}:=p\sb{1}\circ\widehat{\xi}}
Thus if we set
\w{\Fud{k}{i}:\gamma\sb{i-k}\otimes\dotsc\otimes\gamma\sb{i}\to
P\sp{k}\MapK(a\sb{i},a\sb{i-k-1})} equal to \w[,]{\zeta\circ\widehat{\psi}}
we see that
$$
\pi\sb{t}\circ\wrh{k-1}{K}\circ\Fud{k}{i}~=~
\di{t}\circ\Fud{k}{i}~=~\hmu{k-t-1,t}(\Fud{k-t-1}{i-t-1}\otimes\Fud{t}{i})
$$
\noindent (see \S \ref{dloop} and \wref[)]{eqhocc} for all \w[,]{0\leq t<k} and
\w[.]{\phi\circ\Fud{k}{i}=\Gud{k}{i}}

Thus by induction we see that any \wwb{n-1}st order chain complex \w{\dL\up{n-1}}
extending \w{\dL\up{0}} lifts along $\phi$ to \w[,]{\dK\up{n-1}} so that
\w{\phi\sb{\ast}} is surjective.

On the other hand, since $\phi$ is a trivial fibration in \w[,]{\mapM} in particular
\w{\lop{n-1}\phi:\lop{n-1}\MapK(a\sb{n+1},a\sb{0})\to
\lop{n-1}\MapL(a\sb{n+1},a\sb{0})} is a trivial fibration in $\eM$, so it
induces an isomorphism
$$
\pi\sb{n-1}\mapM(\gamma\sb{1}\dotsc\gamma\sb{n+1},
\MapK(a\sb{n+1},a\sb{0}))\xra{\cong}
\pi\sb{n-1}\mapM(\gamma\sb{1}\dotsc\gamma\sb{n+1},
\MapL(a\sb{n+1},a\sb{0}))
$$
\noindent by SM7. Thus if \w{[\lra{\phi\sb{\ast}\dK}]=[\lra{\phi\sb{\ast}\dK'}]} in
\w[,]{\pi\sb{n-1}\mapM(\gamma\sb{1}\dotsc\gamma\sb{n+1},\
\MapL(a\sb{n+1},\,a\sb{0}))} then \w{[\lra{\dK}]=[\lra{\dK'}]} in
\w[.]{\pi\sb{n-1}\mapM(\gamma\sb{1}\dotsc\gamma\sb{n+1},\
\MapK(a\sb{n+1},\,a\sb{0}))}

We can see directly that
\w{\lra{\dL\up{n-1}}} vanishes if and only if it lifts to
\w[,]{\Fud{n}{n+1}:\gamma\sb{0}\otimes\dotsc\otimes\gamma\sb{n+1}\to
P\sp{n+1}\MapK(a\sb{n+1},a\sb{0})} this happens if and only if the corresponding
value \w{\lra{\dK\up{n-1}}} vanishes, too\vsm.

\noindent\textbf{(b)} \  Now assume that  \w{\phi\up{0}:\dK\up{0}\to\dL\up{0}} is
an arbitrary weak equivalence, but that \w{\dK\up{0}} and \w{\dL\up{0}} are
both fibrant and cofibrant. Factoring
 \w{\phi\up{0}} as a trivial cofibration followed by a trivial fibration, by (a) it
suffices to assume that \w{\phi\up{0}} is a trivial cofibration. This implies that
we have a lifting as indicated in the diagram of $\eM$-categories
\mydiagram[\label{eqthirdlift}]{
\dK\up{0} \ar@{=}[rr] \ar@{^{(}->}[d]^{\phi}_{\simeq} && \dK\up{0} \ar@{->>}[d] \\
\dL\up{0} \ar[rr] \ar@{-->}[rru]^{\rho} && \ast
}
\noindent using Theorem \ref{tcanmc}. Thus by \cite[Proposition 1.2.8]{HovM},
$\phi$ is a homotopy equivalence (with strict left inverse $\rho$). Therefore,
if \w{H:\dL\up{0}\to(\dL\up{0})\sp{\eI}} is a right homotopy
\w{\phi\circ\rho\sim\Id} into a path object for \w{\dL\up{0}} in \w{\MCat}
(cf.\ \cite[I, \S 1]{QuiH}), the two trivial fibrations
\w{d\sb{0},d\sb{1}:(\dL\up{0})\sp{\eI}\epic\dL\up{0}}
induce the required bijection by (a)\vsm.

\noindent\textbf{(c)} \  Finally, if  \w{\phi:\dK\up{0}\to\dL\up{0}}
is any Dwyer-Kan equivalence, with cofibrant replacements
\w{\psi:\widehat{\dK}\up{0}\epic\dK\up{0}} and
\w{\xi:\widehat{\dL}\up{0}\epic\dL\up{0}} in \w{\MCat} (so both $\psi$ and $\xi$
are trivial fibrations), we have a lifting
\mydiagram[\label{eqfourthlift}]{
\ast \ar@{=}[rrrr] \ar@{^{(}->}[d] &&&&
\widehat{\dL}\up{0} \ar@{->>}[d]^{\xi}_{\simeq} \\
\widehat{\dK}\up{0} \ar@{->>}[rr]^{\psi}_{\simeq} \ar@{-->}[rrrru]^{\rho} _{\simeq}&&
\dK\up{0} \ar[rr]^{\phi}_{\simeq}  && \dL\up{0}
}
\noindent where $\rho$ is a Dwyer-Kan equivalence between fibrant and cofibrant
$\eM$-categories, so it induces a bijection as required by (b), while
$\psi$ and $\xi$ are trivial fibrations in \w[,]{\MCat} so they induce the
required bijections by (a).  Since the lower right quadrangle in \wref{eqfourthlift}
commutes, $\phi$ also induces a bijection as required.
\end{proof}

\begin{defn}\label{dhchaincx}
Given a path model category $\eM$ with core $\cI$, let $\eC$ be a (small)
subcategory of \w{\MCat} consisting of fibrant $0$-th order chain complexes
of length \w{N=n+1} for \w[.]{\Gamma\subseteq\cI} If $\sim$ is the equivalence
relation on $\eC$ generated by Dwyer-Kan equivalences, let
\w[.]{\Ho\sp{\Gamma}\eC:=\eC/\sim} An equivalence class in \w{\Ho\sp{\Gamma}\eC}
will be called a \emph{homotopy chain complex} for $\Gamma$.
\end{defn}

\begin{example}\label{eghchaincx}
Our motivating example is when $\eC$ is an $\eM$-subcategory of a model category
\w[,]{\eC'} whose weak equivalences \w{f:X\to Y} between fibrant objects are
maps inducing an isomorphism
in \w{f\sb{\ast}:\pi\sb{\ast}\mapM(\gamma,\MapCp(Z,X))\to
\pi\sb{\ast}\mapM(\gamma,\MapCp(Z,Y))} for every cofibrant \w{Z\in\eC'} and every
\w[.]{\gamma\in\cI}  Examples include those of \S\S \ref{egtop}-\ref{egchaincx}
with $\cI$ as in \S \ref{egcore}.

In this case a homotopy chain complex $\Lambda$ of length \w{n+1} in
\w{\Ho\sp{\Gamma}\eC} is represented by a sequence of elements
\begin{myeq}\label{eqnhochaincx}
\varphi\sb{i}\in[a\sb{i},\,a\sb{i-1}]\sb{\gamma\sb{i}}~\cong~
\pi\sb{0}\mapM(\gamma\sb{i},\,\MapC(a\sb{i},\,a\sb{i-1}))\hspace{10mm}
(i=1,\dotsc,n+1)
\end{myeq}
\noindent such that
$$
\nu\sb{\ast}(\varphi\sb{i-1},\varphi\sb{i})=0\hspace{5mm}\text{in}\hspace{3mm}
[a\sb{i},\,a\sb{i-2}]\sb{\gamma\sb{i-1}\otimes\gamma\sb{i}}\hspace{10mm}
(i=2,\dotsc n+1)~,
$$
\noindent in the notation of \S \ref{dcorehtpy}.

In particular, when \w[,]{\cI=\{1\}} $\Lambda$ may be described by
a diagram:
\begin{myeq}\label{eqnhochain}
a\sb{n+1}~\xra{\varphi\sb{n+1}}~a\sb{n}~\xra{\varphi\sb{n}}~
a\sb{n-2}~\to~ \dotsc~\to~a\sb{1}~\xra{\varphi\sb{1}}~a\sb{0}~,
\end{myeq}
\noindent in \w{\ho\eC'} such that \w{\varphi\sb{i-1}\circ\varphi\sb{i}=0}
for \w[\vsm.]{i=2,\dotsc n+1}

However, in the context of Massey products (cf.\ \S \ref{smassey}), we do not have
such a model category \w{\eC'} available.  In this case, we let $\eC$ be a
set of DGAs over $R$ with a given homology algebra, \w{\Gamma=\cI\sb{R}}
as in \S \ref{egcore}(b), and a homotopy chain complex $\Lambda$
in \w{\Ho\sp{\Gamma}\eC} is a quasi-isomorphism class of DGAs in $\eC$.
\end{example}

\begin{defn}\label{dhhtb}
Given a path model category $\eM$ with core $\cI$, a category $\eC$ as
in \S \ref{dhchaincx} for \w[,]{\Gamma\subseteq\cI} and a homotopy chain complex
$\Lambda$ of length \w{n+1} for $\Gamma$, the corresponding \emph{$n$-th order
Toda bracket} \w{\llrra{\Lambda}} is defined to be
\w{\llrra{\dK\up{0}}\subseteq
\pi\sb{n-1}\mapM(\gamma\sb{1}\otimes\dotsc\otimes\gamma\sb{n+1},\
\MapK(a\sb{n+1},\,a\sb{0}))}
for some representative \w{\dK\up{0}} of $\Lambda$.
\end{defn}

\begin{remark}\label{rhhtb}
By Theorem \ref{tdke}, \w{\llrra{\Lambda}} is well-defined.
\end{remark}

\begin{mysubsection}{Massey products in DGAs}
\label{smpdga}
Since \w{\ChR} is a model category, we can consider higher Toda brackets
for a differential graded algebra $\bA$, as in \S \ref{smassey}
(we think of $\bA$ as a chain complex, rather than a cochain complex, but
since we allow arbitrary $\bZ$-grading, this is no restriction).

A chain complex $\Lambda$ of length \w{n+1} in \w{\ho\bA} consists of a sequence
\w{(\gamma\sb{i})\sb{i=1}\sp{n+1}} of homology classes in \w[,]{H\sb{\ast}\bA}
with \w{\gamma\sb{i}\cdot\gamma\sb{i+1}=0} for \w[.]{i=1,\dotsc,n}
If we choose an $n$-th order chain complex (that is, a DGA \w[)]{\bA} realizing
$\Lambda$, as above, we obtain the element given by \wref{eqcohelts} in
\w[.]{\lop{n-1}\bA} However, because we are working over \w{\ModR} we can define the
identification \w{\lop{n-1}\bA\cong\Omega\sp{n-1}\bA} using the Dold-Kan
equivalence (essentially, by the homotopy addition theorem \wh cf.\
\cite{MunkrS}), and thus obtain the value
\begin{myeq}\label{eqmasseych}
\sum\sb{t=1}\sp{n}\ (-1)\sp{t}\
H\sp{t}\sb{i-k+t}\cdot H\sp{n-t}\sb{i}~\in A\sb{j+n}
\end{myeq}
\noindent in \w[,]{\Omega\sp{n-1}\bA} which is readily seen to be a
\wwb{j+n-1}-cycle for \w[.]{j\DEF \sum\sb{t=1}\sp{n}\,m\sb{t}}

By comparing this formula with the classical definition of the higher Massey
product (see, e.g., \cite[(V.4)]{TanrH}), we find:
\end{mysubsection}

\begin{prop}\label{pmassey}
The higher Toda brackets in a differential graded algebra \w{\bA} are identical
with the usual higher Massey products.
\end{prop}

%
%
\sect{Toda brackets for chain complexes}
\label{ctbcc}

We now study Toda brackets in the category \w{\ChRz} of non-negatively graded
chain complexes over a hereditary ring $R$, such as $\bZ$. It turns out that
in this case even ordinary Toda brackets have a finer ``homological'' structure,
which we describe.

\begin{mysubsection}{Chain complexes over hereditary rings}
\label{ch-cx-homo}
Since $R$ is hereditary, if \w{\Qz(G)} is a functorial free cover of an
$R$-module $G$, we have a projective presentation
$$
0 \to \Qb(G) \xra{\alpha\sp{G}} \Qz(G) \xra{r}G \to 0~,
$$
\noindent where \w[.]{\Qb(G)\DEF\Ker(r)}

We then define the $n$-th \emph{Moore complex} \w{\bM{G}{n}} for an
$R$-module $G$ to be the chain complex with
\w[,]{(M(G,n))\sb{n+1}\DEF\Qb(G)} \w[,]{(M(G,n))\sb{n}\DEF\Qz(G)} and $0$
otherwise, with \w[.]{\partial\sb{n+1}=\alpha\sp{G}} This yields a functor
\w{\whC\colon\grMz\to\ChRz} with
\begin{myeq}\label{eqwhc}
\whC(\cE)~\DEF~\bigoplus\sb{n\geq 0}\ \bM{E\sb{n}}{n}~.
\end{myeq}

Recall that \w{\ChRz} has a model structure in which
quasi-isomorphisms are the weak equivalences, and a chain complexes is cofibrant
if and only if it is projective in each dimension (see \cite[\S 2.3]{HovM}).
Because $R$ is hereditary, any \w{\bA \in \ChRz} is uniquely determined up
to weak equivalence by the graded $R$-module \w{H\sb{\ast}\bA}
(cf.\ \cite[Theorem 3.4]{DoldH}).

Therefore, if we enrich \w{\grMz} over \w{\ChR} by setting
$$
\uHom(\cE,\cF)~\DEF~\hHom(\whC(\cE), \whC(\cF))
$$
\noindent (see \S \ref{egchaincx}), \w{\whC} becomes an
enriched embedding, and in fact:
\end{mysubsection}

\begin{lemma}\label{ldkeq}
The functor \w{\whC\colon\grMz\to\ChRz} is a Dwyer-Kan equivalence over
\w[.]{\ChR}
\end{lemma}

Since the right-hand side of \wref{eqwhc} is a coproduct,
we see that \w{\uHom(\cE,\cF)} naturally splits as a product
\begin{myeq}\label{eqmapdecomp}
\prod\sb{n\geq 0}\,
\big(\hHom(\bM{E\sb{n}}{n},\,\bM{F\sb{n}}{n})
\times \hHom(\bM{E\sb{n}}{n},\,\bM{F\sb{n+1}}{n+1})\big)\times P,
\end{myeq}
\noindent where $P$ is a product of similar terms, but with \w[.]{H\sb{0}P=0}
Moreover, since
\begin{myeq}\label{eqhomext}
\pih{\bM{E}{n}}{\bM{F}{n}}\cong\Hom\sb{R}(E,\,F)\ \text{and}\
\pih{\bM{E}{n}}{\bM{F}{n+1}}\cong\Ext\sb{R}(E,\,F),
\end{myeq}
\noindent we see that \wref{eqmapdecomp} is an enriched version of
the Universal Coefficient Theorem for chain complexes, stating that for
chain complexes over a hereditary ring $R$ there is a (split)
short exact sequence:
$$
0\to\displaystyle{\prod\sb{n>0}}\ \Ext\sb{R}(H\sb{n-1}\bA, H\sb{n}\bB)\to
\pih{\bA}{\bB}\to\displaystyle{\prod\sb{n\geq 0}}\
\Hom\sb{R}(H\sb{n}\bA,H\sb{n}\bB)\to 0
$$
\noindent (cf.\ \cite[Corollary 10.13]{DoldA}). Note that in our version
for \w[,]{\grMz} the splitting is natural!

\begin{notn}\label{nmatrix}
From \wref{eqmapdecomp} we see that there are two kinds of indecomposable
maps of chain complexes (and their nullhomotopies) (see \wref[):]{eqhemaps}
\begin{enumerate}
\renewcommand{\labelenumi}{(\alph{enumi})~}
\item `$\Hom$-type' maps \w[,]{H(f):\bM{E}{n}\to\bM{F}{n}} determined by
$$
f\bnz\colon\Qz(E)\to\Qz(F)\hsp \text{and}\hsm f\bnb\colon\Qb(E)\to\Qb(F)~.
$$
\noindent A nullhomotopy \w{H(S):H(f)\sim 0}
is given by \w[,]{S\znb\colon\Qz(E)\to\Qb(F)} the factorization of
\w{f\znz} through \w[.]{\Qb(F)\hra\Qz(F)} If it exists, it is unique.
\item `$\Ext$-type' maps
$$
E(f):\bM{E}{n}\to\bM{F'}{n+1}~,
$$
\noindent determined by \w[.]{f\znb\colon\Qb(E)\to\Qz(F')} A nullhomotopy
\w{E(S):E(f)\sim 0} is given by \w{S\znz\colon\Qz(E)\to\Qz(F')} and
\w[.]{S\bnb\colon\Qb(E)\to\Qb(F')}
\end{enumerate}
\mydiagram[\label{eqhemaps}]{
& & & && \Qb(F\sb{n+1}) \ar@{^{(}->}[d] \\
\mathbf{H(f)}:& \Qb(E\sb{n})\ar@{^{(}->}[d]\ar[r]\sp{f\bnb} &
\Qb(F\sb{n})\ar@{^{(}->}[d]
& \mathbf{E(f)}:&
\Qb(E\sb{n})\ar@{^{(}->}[d]\ar[r]\sp{f\znb} \ar@{.>}[ru]\sp{S\bnb} &
\Qz(F\sb{n+1}) \\
& \Qz(E\sb{n}) \ar@{.>}[ru]\sp{S\znb} \ar[r]\sb{f\znz} & \Qz(F\sb{n}) &&
\Qz(E\sb{n}) \ar@{.>}[ru]\sb{S\znz}&
}
\end{notn}

\begin{mysubsection}{Secondary chain complexes in \ww{\grMz}}
\label{ssccc}
In light of the above discussion, we see that any secondary chain complex
\begin{myeq}\label{eqssccc}
\whC(\cE)~\xra{f}~\whC(\cF)~\xra{g}~\whC(\cG)~\xra{h}~\whC(\ccH)
\end{myeq}
\noindent in the \ww{\ChR}-enriched category \w{\grMz} is a direct sum of secondary
chain complexes of one of the following four elementary forms:

\begin{myeq}\label{eqehheh}
\begin{split}
\bM{E\sb{n}}{n}&~\xra{H(f)\top E(f)}~
\bM{F\sb{n}}{n}\oplus\bM{F\sb{n+1}}{n+1}~\xra{H(g)\bot E(g)}~\\
&\bM{G\sb{n+1}}{n+1}~\xra{H(h)}~\bM{H\sb{n+1}}{n+1}~,
\end{split}
\end{myeq}

\begin{myeq}\label{eqehhee}
\begin{split}
\bM{E\sb{n}}{n}&~\xra{H(f)\top E(f)}~
\bM{F\sb{n}}{n}\oplus\bM{F\sb{n+1}}{n+1}~\xra{H(g)\bot E(g)}~\\
&\bM{G\sb{n+1}}{n+1}~\xra{E(h)}~\bM{H\sb{n+2}}{n+2}~,
\end{split}
\end{myeq}

\begin{myeq}\label{eqhehhe}
\begin{split}
\bM{E\sb{n}}{n}& ~\xra{H(f)}~\bM{F\sb{n}}{n}~\xra{H(g)\top E(g)}\\
&\bM{G\sb{n}}{n}\oplus\bM{G\sb{n+1}}{n+1}~\xra{H(h)\bot E(h)}~
\bM{H\sb{n+1}}{n+1}
\end{split}
\end{myeq}

\begin{myeq}\label{eqeehhe}
\begin{split}
\bM{E\sb{n}}{n}& ~\xra{E(f)}~\bM{F\sb{n+1}}{n+1}~\xra{H(g)\top E(g)}\\
&\bM{G\sb{n+1}}{n+1}\oplus\bM{G\sb{n+2}}{n+2}~\xra{H(h)\bot E(h)}~
\bM{H\sb{n+2}}{n+2}
\end{split}
\end{myeq}

Two additional hypothetical forms, namely:
\begin{enumerate}
\renewcommand{\labelenumi}{(\roman{enumi})~}
\item \w{\bM{E\sb{n}}{n}~\xra{H(f)}~\bM{F\sb{n}}{n}~\xra{H(g)}~
\bM{G\sb{n}}{n}~\xra{H(h)}~\bM{H\sb{n}}{n}}
\item \w{\bM{E\sb{n}}{n}~\xra{E(f)}~\bM{F\sb{n+1}}{n+1}~\xra{E(g)}~
\bM{G\sb{n+2}}{n+2}~\xra{E(h)}~\bM{H\sb{n+3}}{n+3}}
\end{enumerate}
in fact are irrelevant to Toda brackets, for dimensional reasons.

Moreover, the four elementary secondary chain complexes may or may not split
further into one of the following six \emph{atomic} forms:
\begin{enumerate}
\renewcommand{\labelenumi}{(\alph{enumi})~}
\item \w{\bM{E\sb{n}}{n}~\xra{H(f)}~\bM{F\sb{n}}{n}~\xra{H(g)}~
\bM{G\sb{n}}{n}~\xra{E(h)}~\bM{H\sb{n+1}}{n+1}} and two similar
cases with a single $E$-term;
\item \w{\bM{E\sb{n}}{n}~\xra{H(f)}~\bM{F\sb{n}}{n}~\xra{E(g)}~
\bM{G\sb{n+1}}{n+1}~\xra{E(h)}~\bM{H\sb{n+2}}{n+2}} and two similar
cases with a single $H$-term.
\end{enumerate}
\end{mysubsection}

\begin{mysubsection}{Secondary Toda brackets in \ww{\grMz}}
\label{sstbc}
By Definition \ref{dhhtb}, a secondary Toda bracket in the \ww{\ChR}-enriched
category \w{\grMz} is associated to a homotopy chain complex $\Lambda$ of
length $3$ in \w{\ho\grMz} as in \wref[.]{eqnhochain} This means that we replace the
actual chain maps in each of the twelve examples of \S \ref{ssccc} by
their homotopy classes: that is, elements in \w{\Hom\sb{R}(E,F)}
or \w[,]{\Ext\sb{R}(E,F')} respectively.

The compositions \w{\Hom(E,F)\otimes\Ext(F,G)\to\Ext(E,G)}
\w{\Ext(E,F)\otimes\Hom(F,G)\to\Ext(E,G)} simply define the functoriality of
\w[,]{\Ext} while \w{\Ext(E,F)\otimes\Ext(F,G)\to\Ext(E,G)} vanishes for
dimension reasons. Nevertheless, the associated Toda bracket may be
non-trivial.

Note that in this case, as in the original construction of Toda in \cite{TodC}
(see also \cite{SpanS}), the subset \w{\llrra{\Lambda}} of
\w{\pih{\susp\cE}{\ccH}} is actually a double coset of the group
$$
 (\Sigma f)\sp{\s}\pih{\susp\cF}{\ccH}+ h\sb{\s}\pih{\susp\cE}{\cG}~,
$$
\noindent so we can think of \w[,]{\llrra{\Lambda}} which we usually
denote simply by \w[,]{\langle h,g,f \rangle} as taking value in the quotient
abelian group
\begin{myeq}\label{eqindet}
\langle h,g,f \rangle~\in~(\Sigma f)\sp{\s}\pih{\susp\cF}{\ccH}~
\backslash~\pih{\susp\cE}{\ccH}~/~h\sb{\s}\pih{\susp\cE}{\cG}~.
\end{myeq}

Thus the elementary examples of \S \ref{ssccc} may be interpreted as
secondary operations in \w[,]{\Ext\sb{R}} defined under certain vanishing
assumptions, and with an explicit indeterminacy (which may be less than
that indicated in \wref{eqindet} in any specific case).

For example, in \wref{eqehhee} (case (e) above), the operation is defined
for elements in the pullback of
$$
  \begin{xy}
 (0,40)*+{\Hom(E\sb{n},F\sb{n}) \otimes \Ext(F\sb{n},G\sb{n+1})
 \otimes \Ext(G\sb{n+1},H\sb{n+2})}="rhs";
 (45,20)*+{\Ext(E\sb{n},G\sb{n+1}) \otimes \Ext(G\sb{n+1},H\sb{n+2})}="image";
 (0,0)*+{\Ext(E\sb{n},F\sb{n+1})\otimes \Hom(F\sb{n+1},G\sb{n+1})
 \otimes \Ext(G\sb{n+1},H\sb{n+2})}="lhs";
{\ar_-{\comp \otimes \Id} "lhs";"image"};
{\ar^-{\comp \otimes \Id} "rhs";"image"};
  \end{xy}
$$
\noindent and takes value in the quotient group
\w[,]{\Ext(E\sb{n},H\sb{n+1})~/~h\sb{\s}\Hom(E\sb{n},G\sb{n})}
where \w{h\sb{\s}\Hom(E\sb{n},G\sb{n})} refers to the image of the given
element \w{h\in\Ext(G\sb{n},H\sb{n+1})} under precomposition with all
elements of \w[.]{\Hom(E\sb{n},G\sb{n})}

It turns out that cases (a) and (d)  are trivial for dimension reasons, but
we shall now provide examples of non-triviality for four of the remaining cases.
\end{mysubsection}

\begin{example}\label{egone}
Consider the homotopy chain complex $\Lambda$ in \w{\ho\grMz} given by
\w[,]{E\sb{0}=\bZ/2} \w[,]{F\sb{0}=\bZ/4} \w[,]{G\sb{0}=\bZ/2} and
\w[,]{H\sb{1}=\bZ/2} with the corresponding maps
\begin{equation*}
\begin{split}
f=2&\in\bZ/2~=~\Hom(E\sb{0},\ F\sb{0})~=~Hom(\bZ/2,\ \bZ/4)\\
g=1&\in\bZ/2~=~\Hom(F\sb{0},\ G\sb{0})~=~\Hom(\bZ/4,\bZ/2)\\
h=2&\in\bZ/2~=~\Ext(F\sb{0},\ H\sb{1})~=~\Ext(\bZ/2,\bZ/2)~.
\end{split}
\end{equation*}

By Remark \ref{rhhtb}, we may choose any cofibrant chain complexes in \w{\ChZ}
to realize $\Lambda$, not necessarily the functorial versions \w[,]{\whC(\cE)}
and so on. In our case we shall use the  following minimal secondary chain complex:
$$
 \begin{xy}
 (110,40)*+{D\sb{2}=\bZ}="b1d";
(110,20)*+{D\sb{1}=\bZ}= "z1d" ;
{\ar@{^{(}->}^{\alpha\sb{1}\sp{\bD}=2}"b1d";"z1d"};
 (70,20)*+{C\sb{1}=\bZ}="b0c";
(70,0)*+{C\sb{0}=\bZ}= "z0c";
{\ar@{^{(}->}^{\alpha\sb{0}\sp{\bC}=2}"b0c";"z0c"};
{\ar^{h\sb{0}\sp{10}=1}"b0c";"z1d"};
 (30,20)*+{B\sb{1}=\bZ}="b0b";
(30,0)*+{B\sb{0}=\bZ}= "z0b";
{\ar@{^{(}->}^{\alpha\sb{0}\sp{\bB}=4}"b0b";"z0b"};
{\ar_{g\sb{0}\sp{00}=1}"z0b";"z0c"};
{\ar@{-->}^{g\sb{0}\sp{11}=2}"b0b";"b0c"};
{\ar@{.>}^{T\sb{0}\sp{11}=1}"b0b";"b1d"};
(-10,20)*+{A\sb{1}= \bZ}="b0a";
(-10,0)*+{A\sb{0}= \bZ}= "z0a";
{\ar@{^{(}->}^{\alpha\sb{0}\sp{\bA}=2}"b0a";"z0a"};
{\ar_{f\sb{0}\sp{00}=2}"z0a";"z0b"};
{\ar@{-->}^{f\sb{0}\sp{11}=1}"b0a";"b0b"};
{\ar@{.>}|(.5){\hole}^(.38){S\sb{0}\sp{01}=1}"z0a";"b0c"};
 \end{xy}
$$

The Toda bracket is given by:
$$
 \xymatrixcolsep{4cm}
 \xymatrix{
 (\Sigma A)\sb{2}=\bZ
\ar@{-->}[r]^-{1}\ar@{^{(}->}[d]^-{-2}&
D\sb{2}=\bZ \ar@{^{(}->}[d]^-{2}\\
 (\Sigma A)\sb{1}=\bZ \ar[r]\sp{-1}& D\sb{1}=\bZ
 }
$$
\noindent The indeterminacy is given by
\begin{equation*}
\begin{split}
&(\Sigma f)\sp{\s}\pih{\susp\cF}{\ccH}+ h\sb{\s}\pih{\susp\cE}{\cG}~\\
=&~\Sigma f\sp{\s}\Hom(F\sb{0},H\sb{1})+ h\sb{\s}\Hom(E\sb{0},G\sb{1})~=~
 2 \cdot (\bZ/2)+ 0~=~0~.
\end{split}
\end{equation*}
\noindent
\noindent Hence the Toda bracket \w{\langle h,g,f \rangle} does not vanish.
\end{example}

\begin{example}\label{egtwo}
Consider the homotopy chain complex in \w{\ho\grMz} given by
\w[,]{E\sb{0}=\bZ/2} \w[,]{F\sb{1}=\bZ/4} \w[,]{G\sb{1}=\bZ/4} and
\w[,]{H\sb{2}=\bZ} with the corresponding maps
\w[,]{f=1\in\bZ/2=\Ext(\bZ/2,\bZ/4)} \w[,]{g=2\in\bZ/4=\Hom(\bZ/4,\bZ/4)} and
\w[.]{h=2\in\bZ/4=\Ext(\bZ/4,\bZ)}

We choose the following associated secondary chain complex:
$$
 \begin{xy}
 (110,40)*+{D\sb{2}= \bZ}= "z2d";
(70,40)*+{C\sb{2}= \bZ}= "b1c";
(70,20)*+{C\sb{1}= \bZ}= "z1c";
{\ar@{^{(}->}^{\alpha\sb{1}\sp{\bC}=4}"b1c";"z1c"};
{\ar^{h\sb{1}\sp{10}=2} "b1c";"z2d"};
(30,40)*+{B\sb{2}= \bZ}= "b1b";
(30,20)*+{B\sb{1}= \bZ}= "z1b";
{\ar@{^{(}->}^{\alpha\sb{1}\sp{\bB}=4}"b1b";"z1b"};
{\ar_{g\sb{0}\sp{00}=2}"z1b";"z1c"};
{\ar@{-->}^{g\sb{0}\sp{11}=2}"b1b";"b1c"};
{\ar@{..>}|(.5){\hole}^(.4){T\sb{1}\sp{00}= 1}"z1b";"z2d"};
(-10,20)*+{A\sb{1}= \bZ}="b0a";
(-10,0)*+{A\sb{0}= \bZ}= "z0a";
{\ar@{^{(}->}^{\alpha\sb{0}\sp{\bA}=2}"b0a";"z0a"};
{\ar_{f\sb{0}\sp{10}=1}"b0a";"z1b"};
{\ar@{..>}_(.4){S\sb{0}\sp{00}= 1}"z0a";"z1c"};
 \end{xy}
$$

The Toda bracket is represented as follows:
$$
 \xymatrixcolsep{4cm}
 \xymatrix{
 (\Sigma A)\sb{2}=\bZ \ar[r]^-{1}\ar@{^{(}->}[d]^-{-2}& D\sb{2}=\bZ \\
 (\Sigma A)\sb{1}=\bZ &
 }
$$
\noindent which is a generator of \w[.]{\Ext(\bZ/2,\bZ)=\bZ/2}
The indeterminacy is
\begin{equation*}
\begin{split}
&(\Sigma f)\sp{\s}\pih{\susp\cF}{\ccH}+ h\sb{\s}\pih{\susp\cE}{\cG}\\
=&~\Sigma f\sp{\s}\Hom(F\sb{0},H\sb{1})+ h\sb{\s}\Hom(E\sb{0},G\sb{1})~=~
1 \cdot 0+ 2 \cdot (\bZ/2)~=~0~.
\end{split}
\end{equation*}
\noindent Hence the Toda bracket \w{\langle h,g,f \rangle} does not vanish.
\end{example}

\begin{example}\label{egthree}
Consider the homotopy chain complex in \w{\ho\grMz} given by
\w[,]{E\sb{0}=\bZ/8} \w[,]{F\sb{1}=\bZ/4} \w[,]{G\sb{1}=\bZ/4} and
\w[,]{H\sb{2}=\bZ} with the corresponding maps
\w[,]{f=1\in\bZ/4=\Hom(\bZ/8,\bZ/4)} \w[,]{g=2\in\bZ/4=\Ext(\bZ/4,\bZ/4)} and
\w[.]{h=1\in\bZ/4=\Ext(\bZ/4,\bZ)}

We may choose the following associated secondary chain complex:
$$
 \begin{xy}
 (110,40)*+{D\sb{2}=\bZ}="z2d";
 (70,40)*+{C\sb{2}=\bZ}="b1c";
 (70,20)*+{C\sb{1}=\bZ}="z1c";
{\ar@{^{(}->}^{\alpha\sb{1}\sp{\bC}=4}"b1c";"z1c"};
{\ar^{h\sb{1}\sp{10}=1}"b1c";"z2d"};
 (30,20)*+{B\sb{1}=\bZ}="b0b";
(30,0)*+{B\sb{0}=\bZ}= "z0b";
{\ar@{^{(}->}^{\alpha\sb{0}\sp{\bB}=4}"b0b";"z0b"};
{\ar^{g\sb{0}\sp{10}=2}"b0b";"z1c"};
(-10,20)*+{A\sb{1}= \bZ}="b0a";
(-10,0)*+{A\sb{0}= \bZ}= "z0a";
{\ar@{^{(}->}_{\alpha\sb{0}\sp{\bA}=8}"b0a";"z0a"};
{\ar_{f\sb{0}\sp{00}=1}"z0a";"z0b"};
{\ar@{-->}_{f\sb{0}\sp{11}=2}"b0a";"b0b"};
{\ar@{..>}^{S\sb{0}\sp{11}=1}"b0a";"b1c"};
 \end{xy}
$$

The Toda bracket is given by:
$$
 \xymatrixcolsep{4cm}
 \xymatrix{
 (\Sigma A)\sb{2}=\bZ \ar[r]^-{1}\ar@{^{(}->}[d]^-{-8}& D\sb{2}=\bZ \\
 (\Sigma A)\sb{1}=\bZ &
 }
$$
\noindent which is a generator of \w[.]{\Ext(\bZ/8,\bZ)=\bZ/8}
The indeterminacy is
$$
 (\Sigma f)\sp{\s}\pih{\susp\cF}{\ccH}+ h\sb{\s}\pih{\susp\cE}{\cG}~=~
f\sp{\s} \Ext(F\sb{0},H\sb{2})+h\sb{\s} \Hom(E\sb{0},G\sb{1})~.
$$
\noindent A generator of
\w{f\sp{\s} \Ext(F\sb{0},H\sb{2})= 1\cdot \Ext(\bZ/4,\bZ)=\bZ/4}
in \w{\Ext(E\sb{0},H\sb{2})} is given by
$$
 \xymatrixcolsep{2.5cm}
 \xymatrix{
 (\Sigma A)\sb{2}=\bZ \ar@{-->}[r]^-{2}\ar@{^{(}->}[d]^-{-8} &
(\Sigma B)\sb{2}=\bZ \ar[r]^-{1}\ar@{^{(}->}[d]^-{-4}  & D\sb{2}=\bZ\\
 (\Sigma A)\sb{1}=\bZ \ar[r]^-{1} & (\Sigma B)\sb{1}=\bZ  &
 }
$$
\noindent while a generator of
\w{h\sb{\s} \Hom(E\sb{0},G\sb{1})= 1 \cdot \Hom(\bZ/8,\bZ/4)=\bZ/4}
in \w{\Ext(E\sb{0},H\sb{2})} is given by
$$
 \xymatrixcolsep{3cm}
 \xymatrix{
(\Sigma A)\sb{2}=\bZ \ar[r]^-{2}\ar@{^{(}->}[d]^-{-8} & D\sb{2}=\bZ\\
(\Sigma A)\sb{1}=\bZ &
 }
$$
\noindent so the total indeterminacy is the subgroup
\w[.]{\bZ/4\subseteq\bZ/8=\Ext(\bZ/8,\bZ)=\Ext(E\sb{0},H\sb{2})}
Since the Toda bracket \w{\langle h,g,f \rangle} is represented by a
generator of this \w[,]{\bZ/8} it does not vanish.
\end{example}

\begin{example}\label{egzero}
Consider the homotopy chain complex in \w{\ho\grMz} given by
\w[,]{E\sb{0}=\bZ/16} \w[,]{F\sb{0}=\bZ/8} \w[,]{F\sb{1}=\bZ/16}
\w[,]{G\sb{1}=\bZ/16} and \w[,]{H\sb{2}=\bZ/16} with the corresponding maps
\w[,]{f=1\in\bZ/8=\Hom(E\sb{0},F\sb{0})}
\w[,]{f\sp{\prime}\in\bZ/16=\Ext(E\sb{0},F\sb{1})}
\w[,]{g=4\in\bZ/8=\Ext(F\sb{0},G\sb{1})}
\w{g\sp{\prime}\in\bZ/16=\Hom(F\sb{1},G\sb{1})} and
\w[.]{h=2\in\bZ/16=\Hom(G\sb{1},H\sb{1})}

We may choose the following associated secondary chain complex:
$$
 \begin{xy}
 (125,55)*+{D\sb{2}=\bZ}="b1d";
 (125,30)*+{D\sb{1}=\bZ}="z1d";
{\ar@{^{(}->}^{\alpha\sb{1}\sp{\bD}=16}"b1d";"z1d"};
 (85,55)*+{C\sb{2}=\bZ}="b1c";
 (85,30)*+{C\sb{1}=\bZ}="z1c";
{\ar@{^{(}->}|(.44){\hole}^(.6){\alpha\sb{1}\sp{\bC}=16}"b1c";"z1c"};
{\ar@{-->}^{h\sb{1}\sp{11}=2}"b1c";"b1d"};
{\ar^{h\sb{1}\sp{00}=2}"z1c";"z1d"};
(40,55)*+{B\sb{2}\sp{\prime}=\bZ}="b1b";
(40,30)*+{B\sb{1}\sp{\prime}=\bZ}="z1b";
(40,25)*+{\oplus};
(40,20)*+{B\sb{1}=\bZ}="b0b";
(40,0)*+{B\sb{0}=\bZ}= "z0b";
{\ar@{^{(}->}|(.45){\hole}^(.6){\alpha\sb{1}\sp{\bB}=16}"b1b";"z1b"};
{\ar^{(g\sp{\prime})\sb{1}\sp{00}=8}"z1b";"z1c"};
{\ar@{-->}^{(g\sp{\prime})\sb{1}\sp{11}=8}"b1b";"b1c"};
{\ar@{^{(}->}^{\alpha\sb{0}\sp{\bB}=8}"b0b";"z0b"};
{\ar_{g\sb{0}\sp{10}=4}"b0b";"z1c"};
{\ar@{..>}@<.5ex>^(.45){T\sb{1}\sp{01}=1}"z1b";"b1d"};
{\ar@{..>}_{T\sb{0}\sp{00}=1}"z0b";"z1d"};
(-5,20)*+{A\sb{1}= \bZ}="b0a";
(-5,0)*+{A\sb{0}= \bZ}= "z0a";
{\ar@{^{(}->}^{\alpha\sb{0}\sp{\bA}=16}"b0a";"z0a"};
{\ar_{f\sb{0}\sp{00}=1}"z0a";"z0b"};
{\ar@{-->}_{f\sb{0}\sp{11}=2}"b0a";"b0b"};
{\ar^{(f\sp{\prime})\sb{0}\sp{10}=1}"b0a";"z1b"};
{\ar@{..>}@<.3ex>@/^1.2pc/^(.42){S\sb{0}\sp{11}=1}"b0a";"b1c"};
 \end{xy}
$$

The Toda bracket is given by:
$$
 \xymatrixcolsep{4.5cm}
 \xymatrix{
 (\Sigma A)\sb{2}=\bZ \ar@{-->}[r]^-{h\sb{1}\sp{11}\circ S\sb{0}\sp{11}-T\sb{1}\sp{01} \circ (f\sp{\prime})\sb{0}\sp{10}=1}
 \ar@{^{(}->}[d]^-{-16}&  D\sb{2}=\bZ\ar@{^{(}->}[d]^-{16} \\
 (\Sigma A)\sb{1}=\bZ \ar[r]^-{-f\sb{0}\sp{00} \circ T\sb{0}\sp{00}=-1}& D\sb{1}=\bZ
 }
$$
\noindent which is a generator of \w[.]{\Hom(\bZ/16,\bZ/16)=\bZ/16}
The indeterminacy is
$$
(\Sigma f)\sp{\s}\pih{\susp\cF}{\ccH}+ h\sb{\s}\pih{\susp\cE}{\cG}~=~
f\sp{\s} \Hom(F\sb{0},H\sb{1})+h\sb{\s} \Hom(E\sb{0},G\sb{1})~.
$$
\noindent A generator of
\w{f\sp{\s} \Hom(F\sb{0},H\sb{1})= 1\cdot \Hom(\bZ/8,\bZ/16)=\bZ/8}
in \w{\Hom(E\sb{0},H\sb{1})} is given by
$$
 \xymatrixcolsep{2.5cm}
 \xymatrix{
 (\Sigma A)\sb{2}=\bZ \ar@{-->}[r]^-{f\sb{0}\sp{11}=2}\ar@{^{(}->}[d]^-{-16} &
(\Sigma B)\sb{2}=\bZ \ar@{-->}[r]^-{-1}\ar@{^{(}->}[d]^-{-8}  &
D\sb{2}=\bZ\ar@{^{(}->}[d]^-{16}\\
 (\Sigma A)\sb{1}=\bZ \ar[r]^-{f\sb{0}\sp{00}=1} & (\Sigma B)\sb{1}=\bZ\ar[r]^-{2} &
D\sb{1}=\bZ
 }
$$
\noindent while a generator of
\w{h\sb{\s} \Hom(E\sb{0},G\sb{1})= 2 \cdot \Hom(\bZ/16,\bZ/16)=\bZ/8}
in \w{\Hom(E\sb{0},H\sb{1})} is given by
$$
 \xymatrixcolsep{3cm}
 \xymatrix{
(\Sigma A)\sb{2}=\bZ \ar@{-->}[r]^-{-2}\ar@{^{(}->}[d]^-{-16} &
D\sb{2}=\bZ\ar@{^{(}->}[d]^-{16}\\
(\Sigma A)\sb{1}=\bZ \ar[r]^-{2} & D\sb{1}=\bZ&
 }
$$
\noindent so the Toda bracket \w{\langle h,g,f \rangle} does not vanish.
\end{example}

\begin{remark}
See \cite[\S 6.12]{BauH} for a calculation relating Toda brackets in topology
with a certain operation in homological algebra.
\end{remark}


\begin{thebibliography}{ABCD}
%
\bibitem[Ad1]{AdSS}
J.F.~Adams,
``On the structure and applications of the {Steenrod} algebra'',
\textit{Comm.\ Math.\ Helv.} \textbf{32} (1958), 180--214.
%
\bibitem[Ad2]{AdHI}
%
J.F.~Adams,
``On the non-existence of elements of {Hopf} invariant one'',\hsm
\textit{Ann.\ Math.\ (2)} \textbf{72} (1960), No.\ 1, pp.\ 20-104.
%
\bibitem[Al]{JCAlexC}
J.C.~Alexander,
``Cobordism Massey products'',\hsm
\textit{Trans.\ AMS} \textbf{166} (1972), pp.~197-214.
%
\bibitem[BT]{BTaimM}
I.K.~Babenko \& I.A.~Ta{\u{\i}}manov,
``Massey products in symplectic manifolds'',
\textit{Mat.\ Sb.} \textbf{191} (2000), pp.~3-44.
%
\bibitem[BJM]{BJMahT}
M.G.~Barratt, J.D.S.~Jones \& M.E.~Mahowald,
``Relations amongst Toda brackets and the Kervaire invariant in dimension
$64$'',\hsm
\textit{J.\ Lond.\ Math.\ Soc.} \textbf{30} (1984), pp.\ 533-550.
%
\bibitem[Ba1]{BauO}
H.-J.~Baues,
\emph{Obstruction Theory on Homotopy Classification of Maps},\hsm
Springer-\-Verlag \textit{Lec.\ Notes Math.} \textbf{628},
Berlin-\-New York, 1977.
%
\bibitem[Ba2]{BauH}
H.-J.~Baues,
\emph{Homotopy type and Homology},\hsm
Oxford Mathematical Monographs, New York, 1996.
%
\bibitem[Ba3]{BauA}
H.-J.~Baues,
\emph{The Algebra of secondary cohomology operations},
Progress in Math. \textbf{247}, Birkh\"auser, 2006
%
\bibitem[BB]{BBlaC}
H.-J.~Baues \& D.~Blanc
``Higher order derived functors and the Adams spectral sequence'',
\textit{J.\ Pure Appl.\ Alg.} \textbf{219} (2015), pp.~199-239.
%
\bibitem[BKS]{BKSchwR}
D.~Benson, H.~Krause, \& S.~Schwede,
``Introduction to realizability of modules over Tate cohomology'',
in \textit{Representations of algebras and related topics},
AMS, Providence, RI, 2005, pp.~81--97.
%
\bibitem[BM]{BMoerH}
C.~Berger \& I.~Moerdijk,
``On the homotopy theory of enriched categories'',\hsm
\textit{Q.\ J.\ Math.} \textbf{64} (2013), pp.~805--846.
%
\bibitem[BJT1]{BJTurH}
D.~Blanc, M.W.~Johnson, \& J.M.~Turner,
``Higher homotopy operations and cohomology'',
\textit{J.\ $K$-Theory} \textbf{5} (2010), 167--200.
%
\bibitem[BJT2]{BJTurHA}
D.~Blanc, M.W.~Johnson, \& J.M.~Turner,
``Higher homotopy operations and Andr\'{e}-Quillen cohomology'',\hsm
\\textit{Adv.\ Math.} \textbf{230} (2012), pp.~777-817.
%
\bibitem[BV]{BVogHI}
J.M.~Boardman \& R.M.~Vogt,
\textit{Homotopy Invariant Algebraic Structures on Topological Spaces},
Springer-\-Verlag \textit{Lec.\ Notes Math.} \textbf{347},
Berlin-\-New York, 1973.
%
\bibitem[D1]{DoldH}
A.~Dold,
``Homology of symmetric products and other functors of complexes'',\hsm
\textit{Ann.\ Math.\ (2)} \textbf{68} 1958 pp.~54--80.
%
\bibitem[D2]{DoldA}
A.~Dold,
\textit{Lectures on algebraic topology}.
Springer-Verlag, Berlin, 1995.
%
\bibitem[EKMM]{EKMMayR}
A.D.~Elmendorf, I.~K\v{r}i\v{z}, M.A.~Mandell, \& J.P.~May,
\textit{Rings, modules, and algebras in stable homotopy theory},
AMS, Providence, RI, 1997.
%
\bibitem[FW]{FWeldM}
D.~Fuchs \& L.L.~Weldon,
``Massey brackets and deformations'',\hsm
\textit{J.\ Pure Appl.\ Alg.} \textbf{156} (2001), pp.~215--229.
%
\bibitem[G]{GranT}
M.~Grant,
``Topological complexity of motion planning and Massey products'',
in \textit{Algebraic topology--old and new}, Polish Acad. Sci. Inst. Math.,
Warsaw, 2009, pp.~193--203.
%
\bibitem[HKM]{HKMarC}
K.A.~Hardie, K.H.~Kamps, \& H.J.~Marcum,
``A categorical approach to matrix Toda brackets'',\hsm
\textit{Trans.\ AMS} \textbf{347} (1995), pp.~4625-4649.
%
\bibitem[He]{HellSH}
A.~Heller,
``Stable homotopy categories'',\hsm
\textit{Bull.\ AMS} \textbf{74} (1968), pp.~28-63.
%
\bibitem[Hi]{PHirM}
P.S.~Hirschhorn,
\textit{Model Categories and their Localizations},\hsm
Math.\ Surveys \& Monographs \textbf{99}, AMS, Providence, RI, 2002.
%
\bibitem[Ho]{HovM}
M.A.~Hovey,
\textit{Model Categories},\hsm
Math.\ Surveys \& Monographs \textbf{63}, AMS, Providence, RI, 1998.
%
\bibitem[HSS]{HSSmiS}
M.A.~Hovey, B.E.~Shipley, J.H.~Smith,
``Symmetric spectra'',
\textit{Jour.\ AMS} \textbf{13} (2000), 149--208
%
\bibitem[Ki]{KimM}
M.~Kim,
``Massey products for elliptic curves of rank $1$'',\hsm
\textit{J.\ AMS} \textbf{23} (2010), pp.~725--747.
%
\bibitem[Kl]{KlauT}
S.~Klaus,
``Towers and Pyramids,  I'',
\textit{Fund.\ Math} \textbf{13} (2001), No. 5, pp.\ 663-683.
%
\bibitem[Kra]{KraiM}
D.P.~Kraines,
``Massey higher products'',\hsm
\textit{Trans.\ AMS} \textbf{124} (1966), 431-449.
%
\bibitem[Kri]{KristS}
L.~Kristensen,
``On secondary cohomology operations'',\hsm
\textit{Math.\ Scand.} \textbf{12} (1963), 57--82.
%
\bibitem[KM]{KMadE}
L.~Kristensen \& I.H.~Madsen,
``On evaluation of higher order cohomology operations'',\hsm
\textit{Math.\ Scand.} \textbf{20} (1967), pp.~114-130.
%
\bibitem[La]{LaurTB}
G.~Laures,
``Toda brackets and congruences of modular forms'',\hsm
\textit{Alg.\ Geom.\ Top.} \textbf{11} (2011), pp.~1893-1914.
%
\bibitem[LS]{LStreL}
P.~Laurence \& E.~Stredulinsky,
``A lower bound for the energy of magnetic fields supported in linked tori'',\hsm
\textit{C.R.\ Acad.\ Sci.\ Paris I, Math.} \textbf{331} (2000), pp.~201--206.
%
\bibitem[Lu]{LuriH}
J.~Lurie,
\textit{Higher Topos Theory},
Princeton U.\ Press, Princeton, 2009.
%
\bibitem[MP]{MPetS}
M.E.~Mahowald \& F.P.~Peterson,
``Secondary operations on the {Thom} class'',\hsm
\textit{Topology} \textbf{2} (1964), pp.\ 367-377
%
\bibitem[MMSS]{MMSShipM}
M.A.~Mandell, J.P.~May, S.~Schwede, \& B.~Shipley,
``Model categories of diagram spectra'',\hsm
\textit{Proc.\ London Math.\ Soc.\ (3)} \textbf{82} (2001), 441-512.
%
\bibitem[Mar]{MargS}
H.R.~Margolis,
\textit{Spectra and the Steenrod Algebra: \ Modules over the Steenrod Algebra
and the Stable Homotopy Category},\hsm
North-Holland, Amsterdam-\-New York, 1983.
%
\bibitem[Mas]{MassN}
W.S.~Massey,
``A new cohomology invariant of topological spaces'',\hsm
\textit{Bull.\ AMS} \textbf{57} (1951), p.\ 74.
%
\bibitem[Mat]{MathP}
M.~Mather,
``Pull-backs in homotopy theory'',\hsm
\textit{Can.\ J.\ Math.} \textbf{28} (1976), pp.\ 225-263.
%
\bibitem[Mau]{MaunC}
C.R.F.~Maunder,
``Cohomology operations of the $N$-th kind'',\hsm
\textit{Proc.\ Lond.\ Math.\ Soc.\ (2)} \textbf{13} (1963), pp.\ 125-154.
%
\bibitem[Mi]{MizuM}
T.~Mizuno,
``On a generalization of Massey products'',\hsm
\textit{J.\ Math.\ Kyoto Univ.} \textbf{48} (2008), pp.~639--659.
%
\bibitem[Mo]{MoriHT}
M.~Mori,
``On higher Toda brackets'',\hsm
\textit{Bull.\ College Sci.\ Univ.\ Ryukyus} \textbf{35} (1983), pp.\ 1-4.
%
\bibitem[Mu]{MunkrS}
J.R.~Munkres,
``The special homotopy addition theorem'',\hsm
\textit{Mich.\ Math.\ J,} \textbf{2} (1953/54), 127-134.
%
\bibitem[N]{NeiP}
J.A.~Neisendorfer,
\textit{Primary homotopy theory},
Mem.\ AMS \textbf{25} AMS, Providence, RI, 1980.
%
\bibitem[PS]{PSteS}
F.P.~Peterson \& N.~Stein,
``Secondary cohomology operations: two formulas'',\hsm
\textit{Amer.\ J.\ Math.} \textbf{81} (1959), pp.\ 231-305.
%
\bibitem[P1]{GPorW}
G.J.~Porter,
``Higher order Whitehead products'',\hsm
\textit{Topology} \textbf{3} (1965), 123-165.
%
\bibitem[P2]{GPorH}
G.J.~Porter,
``Higher products'',\hsm
\textit{Trans.\ AMS} \textbf{148} (1970), 315-345.
%
\bibitem[Q]{QuiH}
D.G.~Quillen,
\textit{Homotopical Algebra},\hsm
Springer-\-Verlag \textit{Lec.\ Notes Math.} \textbf{20},
Berlin-\-New York, 1963.
%
\bibitem[Re]{RetaL}
V.S.~Retakh,
``Lie-Massey brackets and $n$-homotopically multiplicative maps of
differential graded Lie algebras'',\hsm
\textit{J.\ Pure \& Appl.\ Alg.} \textbf{89} (1993) No.\ 1-2, pp.\ 217-229.
%
\bibitem[Ri]{RizzI}
C.~Rizzi,
``Infinitesimal invariant and Massey products'',\hsm
\textit{Man.\ Math.} \textbf{127} (2008), pp.~235--248.
%
\bibitem[Sa]{SagaU}
S.~Sagave,
``Universal Toda brackets of ring spectra'',\hsm
\textit{Trans.\ AMS} \textbf{360} (2008), pp.~2767-2808.
%
\bibitem[SS]{SShipE}
S.~Schwede \& B.~Shipley,
``Equivalences of monoidal model categories'',\hsm
\textit{Alg.\ Geom.\ Topol.} \textbf{3} (2003), 287--334.
%
\bibitem[Sp1]{SpanS}
E.H.~Spanier,
``Secondary operations on mappings and cohomology'',\hsm
\textit{Ann.\ Math.\ (2)} \textbf{75} (1962) No.\ 2, pp.\ 260-282.
%
\bibitem[Sp2]{SpanH}
E.H.~Spanier,
``Higher order operations'',\hsm
\textit{Trans.\ AMS} \textbf{109} (1963), pp.\ 509-539.
%
\bibitem[St]{SteCC}
N.E.~Steenrod,
``A convenient category of topological spaces'',\hsm
\textit{Michigan Math.\ J,} \textbf{14} (1967), pp.~133-152.
%
\bibitem[Ta]{TanrH}
D.~Tanr\'{e},
\textit{Homotopie Rationelle: Mod\`{e}les de Chen, Quillen, Sullivan},\hsm
Springer-\-Verlag \textit{Lec.\ Notes Math.} \textbf{1025},
Berlin-\-New York, 1983.
%
\bibitem[To1]{TodG}
H.\  Toda,
``Generalized {Whitehead} products and homotopy groups of spheres'',\hsm
\textit{J.\ Inst.\ Polytech.\ Osaka City U., Ser.\ A, Math.} \textbf{3}
(1952), pp.\ 43-82.
%
\bibitem[To2]{TodC}
H.\ Toda,
\textit{Composition methods in the homotopy groups of spheres},\hsm
Adv.\ in Math.\ Study \textbf{49}, Princeton U. Press, Princeton, 1962.
%
\bibitem[V]{VogtCC}
R.M.~Vogt,
``Convenient categories of topological spaces for homotopy theory'',\hsm
\textit{Arch.\ Math.\ (Basel)} \textbf{22} (1971), pp.~545-555.
%
\bibitem[W]{GWalkL}
G.~Walker,
``Long Toda brackets'',
in \textit{Proceedings of the Advanced Study Institute on Algebraic
Topology, Vol. III (Aarhus, 1970)}, Aarhus, 1970, 612--631.
%
\end{thebibliography}
\end{document}